\documentclass[journal,10pt]{IEEEtran}
\usepackage{mathtools}
\usepackage{amsfonts}
\usepackage{amsmath}
\usepackage{enumitem}
\usepackage{amsthm}
\usepackage{bm} 
\usepackage{graphicx}
\usepackage{cite}
\usepackage{color}
\graphicspath{ {./} }
\DeclarePairedDelimiter\ceil{\lceil}{\rceil}

\usepackage{algorithm,algpseudocode}%
\algnewcommand{\Initialize}[1]{%
	\State \textbf{initialize}
	\Statex \hspace*{\algorithmicindent}\parbox[t]{.8\linewidth}{\raggedright #1}
}
\usepackage{subfigure}
\newtheorem{theorem}{Theorem}[section]

\makeatletter
\newtheorem{definition}{Definition}

\begin{document}
\title{An adaptive approach to real-time estimation of vehicle sideslip, road bank angles and sensor bias}
\author{Yi-Wen Liao~and~Francesco Borrelli,~\IEEEmembership{Fellow,~IEEE}
\thanks{Yi-Wen Liao and Francesco Borrelli are with the Department
of Mechanical Engineering, University of California, Berkeley, 94720, USA. e-mail: ywliao@berkeley.edu, fborrelli@berkeley.edu.}}

\markboth{IEEE TRANSACTIONS ON VEHICULAR TECHNOLOGY, ACCEPTED AS A PAPER}%
{Shell \MakeLowercase{\textit{et al.}}: Bare Demo of IEEEtran.cls for Journals}
\maketitle
\begin{abstract}
Robust estimation of vehicle sideslip angle is essential for stability control applications. However, the direct measurement of sideslip angle is expensive for production vehicles. This paper presents a novel sideslip estimation algorithm which relies only on sensors available on passenger and commercial vehicles. The proposed method uses both kinematics and dynamics vehicle models to construct extended Kalman filter observers. The estimate relies on the results provided from the dynamics model observer where the tire cornering stiffness parameters are updated using the information provided from the kinematics model observer.
The stability property of the proposed algorithm is discussed and proven. Finally, multiple experimental tests are conducted to verify its performance in practice. The results show that the proposed approach provides smooth and accurate sideslip angle estimation. In addition, our novel algorithm provides reliable estimates of bank angles and lateral acceleration sensor bias. 
\end{abstract}
\IEEEpeerreviewmaketitle
\section{Introduction}
\IEEEPARstart{A} number of active safety features have been introduced in the automotive industry in the past 30 years to prevent accidents such as braking assistance, traction and electronic stability control systems \cite{pilutti1998vehicle,tseng1999development,hara1991traction,wilson1998driver}. The main goals of these systems are to maintain vehicle stability and to improve vehicle handling. To implement these functions, vehicle states, parameters and road conditions need to be measured or estimated. Among all of these, sideslip angle, the angle between the longitudinal direction of the vehicle and the velocity vector, is one of the most important variables which heavily influences vehicle dynamics and is required by a number of active safety controllers. 
Although it can be directly measured by sensors such as optical sensors \cite{lee2004robust} or GPS sensors \cite{bevly2006integrating,bevly2000use}, these solutions are not implemented by OEMs because of cost and reliability. 
Therefore, the estimation of sideslip angle based on the sensors available in production vehicles is an important topic that has been widely discussed in the literature \cite{selmanaj2017robust,farrelly1996estimation,strano2018constrained,dakhlallah2008tire,aoki2004robust,li2014variable,shao2016nonlinear,you2009new,liu1998state,grip2009vehicle,coy2014decision,peng1experimental,boada2016vehicle,kang2016vehicle,chen2008sideslip,piyabongkarn2009development,de2017real}. Most of the approaches in the literature are model-based and can be classified into three main categories: kinematics model-based, dynamics model-based and a combination of the two.

The kinematics model-based approach proposed in \cite{selmanaj2017robust,farrelly1996estimation} constructs an observer based on the longitudinal and lateral translation kinematics of a point mass model. This method has the advantage of not requiring the vehicle parameters, tire model and road friction coefficient. It can provide an accurate sideslip estimate in a number of cases. However, it suffers from a drifting issue in small yaw rate maneuvers and the estimated result is sensitive to disturbance and measurement noise such as bank angles or sensor bias when the longitudinal and lateral accelerations are small.
The dynamics model-based approach constructs an advanced state observer (i.e. an extended Kalman filter or an unscented Kalman filter) by using a bicycle model or its variations \cite{strano2018constrained,dakhlallah2008tire,aoki2004robust,li2014variable,shao2016nonlinear,you2009new,liu1998state}. These models consider the effect of forces applied to vehicle mass and rotation inertia which provides a relatively robust estimate to acceleration measurement noise compared to the one from the kinematics model. However, an estimation bias is often observed due to the model uncertainties associated to variations of vehicle mass and the tire cornering stiffnesses. Existing literature has also focused on developing algorithms for estimating the sideslip angle and vehicle model system parameters simultaneously. In \cite{shao2016nonlinear,you2009new,liu1998state,grip2009vehicle}, Lyapunov-based observers have been proposed for the tire cornering stiffness identification. Although these techniques can improve the estimation results, they require persistent input excitations and the adaptation performance becomes worse beyond the linear tire model region. Alternative studies have proposed learning-based techniques to assist the traditional adaptation methods \cite{coy2014decision,peng1experimental,boada2016vehicle}. However, the estimated performance is hard to validate in the region when data is limited.
The third category of algorithms tries to merge kinematics and dynamics models into a hybrid solution \cite{kang2016vehicle,chen2008sideslip,piyabongkarn2009development,de2017real}. The concept is to switch between these two estimators and to exploit their respective advantages. This method might look attractive, however the switching will cause a discontinuity in the sideslip estimation. 
Motivated by the idea of the hybrid solution~\cite{kang2016vehicle,chen2008sideslip,piyabongkarn2009development,de2017real} and parameter adaptation, in this paper, we develop a novel sideslip estimation algorithm which only relies on the dynamics model for the estimation but utilizes the strength of the kinematics observer to adapt the front and rear tire cornering stiffnesses. 
In this way, we maintain the advantage of the dynamics model-based observer and further improve the estimator performance in the nonlinear tire region.
In addition, the proposed approach does not need rich input excitation as required in traditional adaptation methods.

The paper is organized as follows. We first introduce commonly used models for kinematics and dynamics-based observer designs. Then, modifications with bank angle and sensor bias effects are considered and included into each of the models. A recursive adaptation algorithm is derived and the stability property is discussed afterwards. Finally, the performance is validated with different scenario tests and compared with existing methods.
\begin{figure}[b]
	\centering
	\includegraphics[width = 3.2 in]{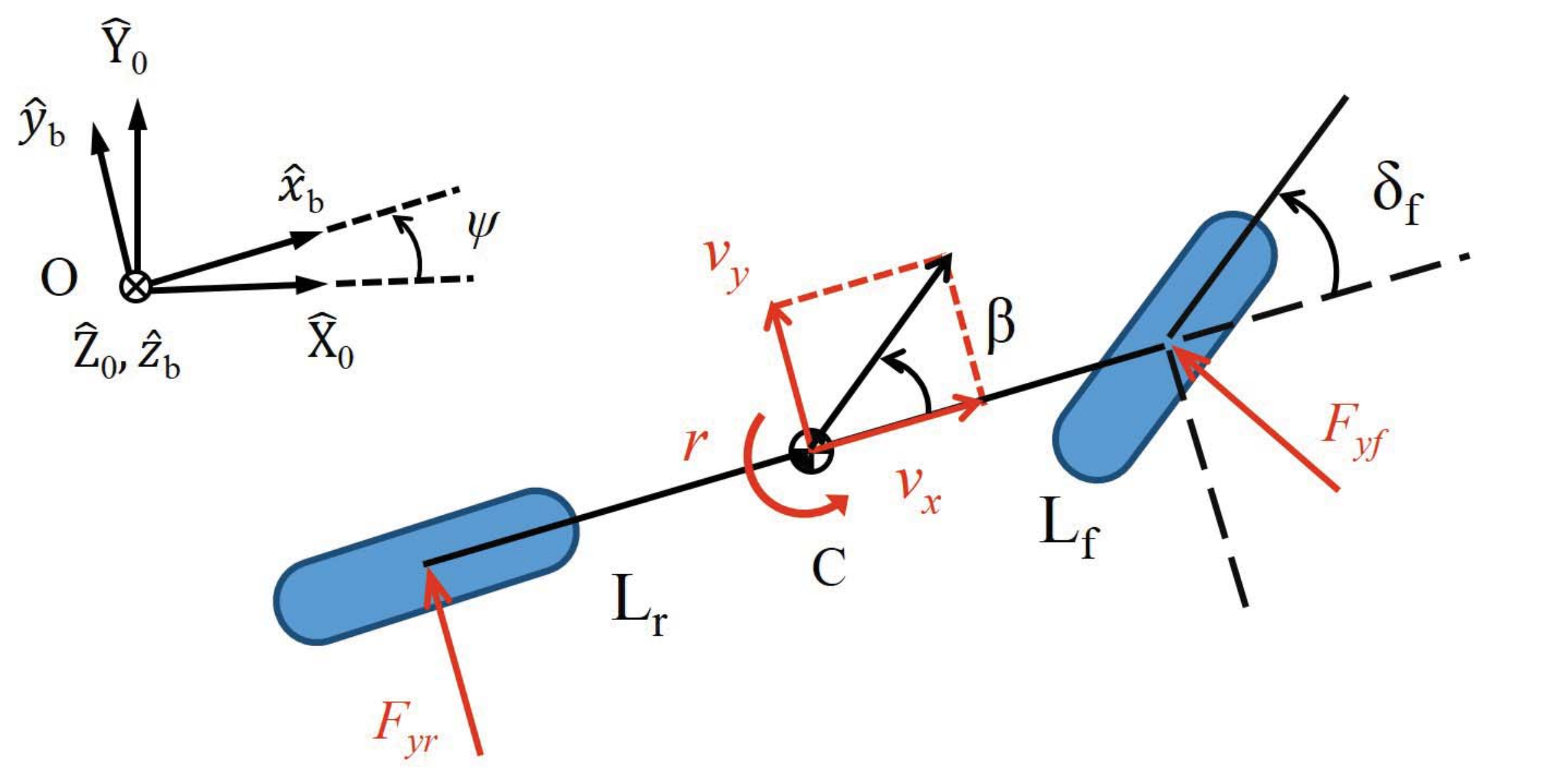}
	\caption{Lateral dynamics for bicycle model.} \label{Fig:dynamicsModel}
\end{figure}
\section{Models for estimation design}
Many of the conventional sideslip angle estimation methods are model-based.
In this section, we are going to introduce two different models which have been frequently used in the literature \cite{peng1experimental}: a 2-DOF point mass rigid body kinematics model and a bicycle dynamics model. Each of them has its own strengths and weaknesses in the observer design. Details are discussed in the following sections. 
\subsection{Kinematics model}
Kinematics is the study of motion which treats the movement of components without considering the forces. To describe the general motion of a rigid body, we first introduce two coordinate systems followed by the ISO convention: one is fixed in the inertial space $\{\hat{X}_0,\hat{Y}_0.\hat{Z}_0\}$ and the other one is fixed to the body $\{\hat{x}_b,\hat{y}_b,\hat{z}_b\}$ (see Fig. \ref{Fig:dynamicsModel}). 
Then, considering the vehicle as a single rigid body constrained to move in the $\hat{X}_0\hat{Y}_0$-plane, the translation motion is described as:
\begin{align}
\label{eqn:kinematicModel}
\begin{split}
a_x = \dot{v}_x-rv_y\\
a_y = \dot{v}_y+rv_x
\end{split}
\end{align}
where $\psi$ is the yaw angle, $r \overset{\Delta}{=} \dot{\psi}$ is the yaw rate of the vehicle. $a_x$ and $v_x$ denote the inertial acceleration and velocity resolved in the longitudinal $\hat{x}_b$-direction. $a_y$ and $v_y$ denote the same physical quantities but resolved in the lateral $\hat{y}_b$-direction. Define the sideslip angle as
\begin{align}
    \beta = \tan^{-1}(v_y/v_x). \nonumber
\end{align}
We write the system (\ref{eqn:kinematicModel}) into a standard state space form as:
\begin{align}
\label{eqn:stateSpaceKin}
\begin{split}
\dot{\mathbf{x}}_k&=A_k(t)\mathbf{x}_k+B_k(t)\mathbf{u}_k\\
\mathbf{y}_k &= C_k(t)\mathbf{x}_k
\end{split}
\end{align}
where $\mathbf{x}_k = [v_x\text{,}~v_y]^T$ is the state vector, $\mathbf{u}_k = [a_x,~a_y]^T$ is the control input vector, $\mathbf{y}_k = v_x$ is the measurement output vector and the system, input and output matrices are
\begin{align}
\begin{split}
A_k(t) &= \begin{bmatrix}
0 &~ r(t)\\
-r(t) &~ 0
\end{bmatrix},~~
B_k(t) = \begin{bmatrix}
1 & 0\\
0 & 1
\end{bmatrix}\\
C_k(t) &= \begin{bmatrix}
1~~0
\end{bmatrix}. 
\end{split}
\label{eq:kinMatrix}
\end{align}

As explained in \cite{ungoren2004study}, using the kinematics model is advantageous as it allows a sideslip angle estimation without requiring vehicle parameters. All we need is $r$, $a_x$, $a_y$ and $v_x$ which can be directly obtained from sensors available in commercial vehicles. However, the estimation is sensitive to sensor noise which is substantial for $a_x$ and $a_y$. Moreover, the convergence of the estimation error can be guaranteed only when yaw rate is not equal to zero. In fact, the system (\ref{eqn:stateSpaceKin}) is not observable when the yaw rate is equal to zero and the poor conditioning of the observability matrix causes a drifting problem. This can be avoided by resetting the estimated states to zero every time when yaw rate is less than a threshold value.
\subsection{Dynamics model}
A variety of dynamics models have appeared in the literature.
The so-called lateral bicycle model shown in Fig. \ref{Fig:dynamicsModel} is a widely used and rather simple model that neglects the coupling of the roll, pitch and longitudinal dynamics.    
By using Newton's law of motion, the lateral dynamics of the bicycle model is described as follows \cite{rajamani2011vehicle}:
\begin{align}
\label{eq:dynamicsnModel}
\begin{split}
ma_y&=m(\dot{v}_y+v_xr) = F_{yf}\cos\delta_f+F_{yr}\\
I_z\dot{r} &= L_fF_{yf}\cos\delta_f-L_rF_{yr}
\end{split}
\end{align}
where $m$ is the vehicle mass, $I_z$ is the equivalent yaw moment of inertia, $\delta_f$ is the front steering angle and $L_f$, $L_r$ are the distance from the vehicle center of gravity (COG) to the front and rear axles. To further simplify the model, we assume small tire slip and front steering angles. Then, the front and the rear lateral tire forces $F_{yf}$, $F_{yr}$ can be approximated by a linear function:   
\begin{align}
\label{eq:tireForces}
\begin{split}
F_{yf}\cos\delta_f\approx F_{yf}  &= C_f\left(\delta_f-\frac{v_y+L_fr}{v_x}\right)\\
F_{yr} &= C_r\left(\frac{-v_y+L_rr}{v_x}\right)
\end{split}
\end{align}
where $C_f$ and $C_r$ are the front and rear tire cornering stiffnesses. Substituting (\ref{eq:tireForces}) into (\ref{eq:dynamicsnModel}), we then obtain a nonlinear model.
Assume that the vehicle is traveling with slowly varying longitudinal velocity. 
At each step, a linearization process will be applied to approximate the nonlinear system (\ref{eq:dynamicsnModel})-(\ref{eq:tireForces}) with a linear time varying system shown as follows:
\begin{align}
\label{eq:dynstate}
\begin{split}
\dot{\mathbf{x}}_d&=A_d(t)\mathbf{x}_d+B_d(t)\mathbf{u}_d \\
\mathbf{y}_d &= C_d(t)\mathbf{x}_d+D_d(t)\mathbf{u}_d 
\end{split}
\end{align}
where $\mathbf{x}_d = [v_y,~r]^T$ is the state vector, $\mathbf{u}_d = \delta_f$ is the control input vector and $\mathbf{y}_d = [a_y,~r]^T$ is the measurement output vector.
\begin{align}
\label{eq:stateDynMatrix}
\begin{split}
A_d(t) &= \begin{bmatrix}
\frac{-C_f-C_r}{mv_x(t)} & -v_x(t)-\frac{L_fC_f-L_rC_r}{mv_x(t)}\\
\frac{-L_fC_f+L_rC_r}{I_{z}v_x(t)} & \frac{-L_f^2C_f-L_r^2C_r}{I_{z}v_x(t)}
\end{bmatrix},\\
C_d(t) &= \begin{bmatrix}
\frac{-C_f-C_r}{mv_x(t)} ~&~ -\frac{L_fC_f-LrC_r}{mv_x(t)}\\
0 ~&~ 1 
\end{bmatrix},\\
B_d(t) &= \begin{bmatrix}
\frac{C_f}{m}\\
\frac{L_fC_f}{I_{z}}
\end{bmatrix},~~~
D_d(t) = \begin{bmatrix}
\frac{C_f}{m}\\
0
\end{bmatrix}.~~~~~
\end{split}
\end{align}

Using the dynamics bicycle model to estimate the sideslip angle has several advantages. First, the estimator can be tuned to be less sensitive to acceleration measurement noise compared to the one based on the kinematics model. Also, drifting and observability issues of the kinematics model are not present. However, the estimated accuracy is affected by the vehicle parameters in the matrices (\ref{eq:stateDynMatrix}).
First, since we use a linear tire model, the sideslip estimate will be accurate only in the linear tire region. Second, compared to $m$, $I_z$, $L_f$ and $L_r$, it is hard to find a good initial condition for the tire cornering stiffness coefficients. To mitigate this issue,  on-line adaptation algorithms have been introduced to identify the cornering stiffness \cite{you2009new,liu1998state}. We will also use this idea in our method.
\section{New sideslip estimation method}
The method proposed in this paper relates to the idea of \cite{chen2008sideslip} which merges the kinematics and dynamics model observers into a hybrid solution. Since the estimated state from the kinematics model is unaffected by the parameter uncertainties, in \cite{chen2008sideslip}, the observer is built to mainly rely on it but will switch to the dynamics model when the absolute value of the yaw rate is less than a threshold value $r_t$ to avoid unobservability and the drifting issue. Although this method addresses the drifting issue, relying on the kinematics model leads to noisy estimates.
Moreover, the switch between the kinematics and dynamics models for the observer often introduces a discontinuous estimate during the transition. 

To overcome these issues and keep the benefits of hybrid models, we propose a new method which is based on a dynamics model but adapts on-line the front and rear tire cornering stiffnesses using information derived from the kinematics model.  
Compared with the traditional adaptation algorithm proposed in \cite{you2009new,liu1998state}, the proposed approach does not need persistent excitation in the control input and also improves the adaptation performance in the nonlinear tire region. For the observer design, we further include the road bank angle disturbance and lateral acceleration sensor bias into the system model in order to minimize possible modeling and estimation errors. This is discussed in next.
\subsection{Augmented Models}
\subsubsection{Dynamics model augmented with the road bank angle and sensor bias}
We consider the bicycle model and include the gravitational force to the lateral dynamics: 
\begin{align}
m(\dot{v}_y+v_xr) = F_{yf}\cos\delta_f+F_{yr}-mg\sin\phi \label{eq: lateral dynamics with bank}
\end{align}
where $\phi$ is the road bank angle with the sign convention shown in Fig. \ref{Fig:bank}. Then, combining the yaw dynamics in (\ref{eq:dynamicsnModel}) with (\ref{eq: lateral dynamics with bank}), we rewrite the first equation of (\ref{eq:dynstate}) as:
\begin{align}
\dot{v}_y = \frac{-(C_f+C_r)}{mv_x}v_y - (v_x+&\frac{L_fC_f-L_rC_r}{mv_x})r \nonumber\\ &~~~~~-g\sin\phi+\frac{C_f}{m}\delta_f. \label{eq:vy_dot}
\end{align}  
The measurement model should also be corrected with the bank disturbance and sensor bias as well. Note that the lateral accelerometer measures the right hand side of (\ref{eq: lateral dynamics with bank}) divided by $m$ and plus the component of gravity in $\hat{y}_b$ direction. We obtain the measurement model of the lateral acceleration as:   
\begin{align}
a_{y}^{sen}& = a_y + g\sin\phi+d = \dot{v}_y+v_xr + g\sin\phi+d \label{eq:aySen}\\
&=\frac{-(C_f+C_r)}{mv_x(t)}v_y-\frac{L_fC_f-L_rC_r}{mv_x(t)}r+d+\frac{C_f}{m}\delta_f.\nonumber
\end{align}   
where $d$ is the sensor bias. By augmenting the system with a constant bank angle disturbance and the sensor bias, the state vector and measurement output are $\mathbf{x}_d = [v_y,~r,~\sin\phi,~~d]^T$ and $\mathbf{y}_d = [a_y^{sen},~r]^T$. The system state space matrices which replace the one in (\ref{eq:stateDynMatrix}) are          
\begin{align}
A_d(t) &= \begin{bmatrix}
\frac{-C_f-C_r}{mv_x(t)} & -v_x(t)-\frac{L_fC_f-L_rC_r}{mv_x(t)} & -g & 0\\
\frac{-L_fC_f+L_rC_r}{I_{z}v_x(t)} & \frac{-L_f^2C_f-L_r^2C_r}{I_{z}v_x(t)} & 0 & 0\\
0 & 0 & 0 & 0\\
0 & 0 & 0 & 0
\end{bmatrix},\nonumber\\
C_d(t) &= \begin{bmatrix}
\frac{-C_f-C_r}{mv_x(t)} & -\frac{L_fC_f-L_rC_r}{mv_x(t)}& 0& 1\\
0 & 1& 0& 0 
\end{bmatrix}, \label{eq:stateDynMatrix2} \\
B_d(t) &= \begin{bmatrix}
\frac{C_f}{m}\\
\frac{L_fC_f}{I_{z}}\\
0\\
0
\end{bmatrix},~~~~~~
D_d(t) = \begin{bmatrix}
\frac{C_f}{m}\\
0
\end{bmatrix}.~~~~~\nonumber
\end{align}
To implement the extended Kalman filter using a digital controller, we further discretized model (\ref{eq:stateDynMatrix2})
using a forward Euler method as:
\begin{align}
\hat{\mathbf{x}}_d[k+1]&=(A_d[k]\Delta t+\text{I}_4)\hat{\mathbf{x}}_d[k]+B_d[k]\Delta t\mathbf{u}_d[k]+\mathbf{w}_d[k]\nonumber\\
\mathbf{y}_d[k] &= C_d[k]\hat{\mathbf{x}}_d[k]+D_d[k]\mathbf{u}_d[k]+\mathbf{v}_d[k]\label{eq:dyn discrete-time model}
\end{align}  
where $\Delta t$ is the sampling period and [$\cdot$] represent the discrete time instant. $\mathbf{w}_d[\cdot]$ and $\mathbf{v}_d[\cdot]$ are the process and measurement noises satisfying the typical assumptions of the extended Kalman filter. 
\begin{figure}
        \centering
		\includegraphics[width = 2.8in]{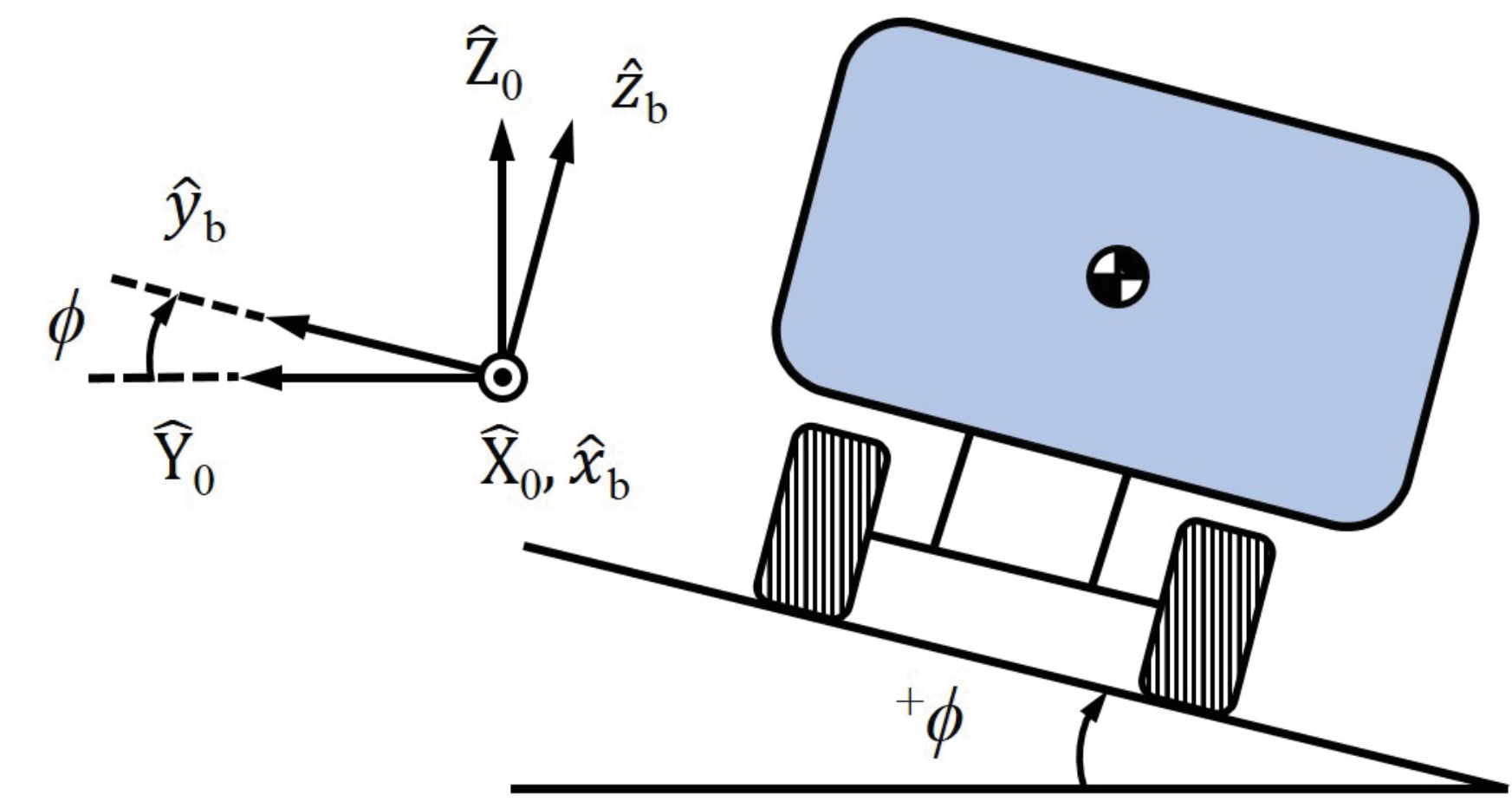}
		\caption{Sign convention for bank angle.} \label{Fig:bank}
\end{figure}
\begin{algorithm*}
	\centering
	\caption{Sideslip Angle Estimation}
	\begin{algorithmic}[1]
		\Initialize{$\hat{\mathbf{x}}_{k}[0] \gets [v_x[0]~v_y[0]~0]^T$,~~$\hat{\mathbf{x}}_{d}[0] \gets [v_y[0]~r[0]~0~~0]^T$,~~
		$\tilde{\theta}^*_0 \gets 0$,~~ $R_0 \gets 0$,~~$\theta_0^* \gets {\theta}^++\tilde{\theta}^*_0$,}
		\State ~~~~$P_k[0] \gets P_{k,0}$,~~$P_d[0] \gets P_{d,0}$~~~~~~~~~~~~~~~~~~~~~//\textit{ initialize prior means and estimate error covariance matrices for} \text{EKF}
		\For{i = 1 to k+1}
		\State $\hat{\mathbf{x}}_{d}[i] \gets$ EKFupdate($\hat{\mathbf{x}}_{d}[i-1],\mathbf{u}_{d}[i-1],\mathbf{u}_{d}[i],\mathbf{y}_{d}[i],P_d[i-1],\theta^*_{i-1}$)~~~~~~// EKF \textit{update with model (\ref{eq:dyn discrete-time model})}
        \State $\hat{\mathbf{x}}_{k}[i] \gets$ EKFupdate($\hat{\mathbf{x}}_{k}[i-1],\mathbf{u}_{k}[i-1],\mathbf{y}_{k}[i],P_k[i-1]$))~~~~~~~~~~~~~~~~~~// EKF \textit{update with model (\ref{eq:kin discrete-time model})}
		\If{$|r_i| \geq r_t$}
		\State $R_{i} = \lambda R_{i-1}+\Phi_{i}\Phi_{i}^T$~~~~~~~~~~~~~~~~~~~~~~~~~~~//\textit{ obtain the input measurement $\Phi_{i}$ from (\ref{eq:regression model})}
		\State $\theta^*_i \gets {\theta}^++$ AdaptationUpdate($R_i, \tilde{\theta}^*_{i-1}$)~~~~~~//\textit{ apply the recursive update law (\ref{eq:recursive adapt law}) for the tire cornering stiffnesses}
		\Else 
		\State $\theta^*_i \gets \theta^*_{i-1}$ 
		\State $\hat{\mathbf{x}}_{k}[i] \gets [v_x[i]~\hat{v}_{y,d}[i]~\sin\hat{\phi}_d[i]]^T$ ~~~~~~~~~~~~//\textit{ update the state estimates of the} \text{EKF} \textit{for model (\ref{eq:kin discrete-time model})}
		\State $P_k[i] \gets \text{diag}(0,P_d[i](1,1),P_d[i](3,3))$ ~~~~~//\textit{ update error covariance matrix of the} \text{EKF} \textit{for model (\ref{eq:kin discrete-time model})}
		\EndIf        
		\State $\beta[i] \gets \tan^{-1}( \hat{v}_{y,d}[i]/v_x[i])$~~~~~~~~~~~~~~~~~~~~~~~//\textit{ calculate sideslip angle}
		\EndFor
	\end{algorithmic}
\end{algorithm*}
\subsubsection{Kinematics model augmented with the road bank angle}
For the kinematics model, we only include the bank angle disturbance into the lateral motion by noting that the system is not observable if the model of the sensor bias is added. The model becomes:
\begin{align}
a_x & = \dot{v}_x-rv_y \nonumber\\
a_y^{sen} & = \dot{v}_y+rv_x+g\sin\phi. \label{eqn:kinematicModel1}
\end{align}  
Then, having the same extended Kalman filter structure shown in (\ref{eq:dyn discrete-time model}), we re-define the state vector $\mathbf{x}_k = [v_x,~v_y,~\sin\phi]^T$, the output vector $\mathbf{u}_k = [a_x,~a_y^{sen}]^T$ and the system state space matrices as
\begin{align}
\label{eq:stateKinMatrix2}
\begin{split}
A_k(t) &= \begin{bmatrix}
~0 & r(t) & 0\\
-r(t) & 0 & -g\\
~0 & 0 & 0
\end{bmatrix},~~~
B_k(t) = \begin{bmatrix}
~1 & 0~\\
~0 & 1~
\end{bmatrix},\\
C_k(t) &= \begin{bmatrix}
~1 ~&~ 0 ~&~ 0~ 
\end{bmatrix}.
\end{split}
\end{align}  
The above model is, again, discretized into:
\begin{align}
\hat{\mathbf{x}}_k[k+1]&=(A_k[k]\Delta t+\text{I}_3)\hat{\mathbf{x}}_k[k]+B_k[k]\Delta t\mathbf{u}_k[k]+\mathbf{w}_k[k]\nonumber\\
\mathbf{y}_k[k+1] &= C_k[k+1]\hat{\mathbf{x}}_k[k+1]+\mathbf{v}_k[k+1].\label{eq:kin discrete-time model}
\end{align} 
\subsection{Adaptation for the tire cornering stiffness}
In the previous section, we have introduced two observer models (\ref{eq:dyn discrete-time model}) and (\ref{eq:kin discrete-time model}). Next, we will show how we merge both observers by using $\hat{v}_{y,k}$, the lateral velocity estimated from (\ref{eq:kin discrete-time model}), to adapt the front and rear tire cornering stiffnesses in the dynamics model (\ref{eq:dyn discrete-time model}). The sideslip estimation will then calculate by using this updated dynamics model. 
\subsubsection{Regression model}
The adaptation is formulated as a regularized weighted least square (RWLS) problem \cite{Tibshirani94regressionshrinkage, eksioglu2011rls}. To build up the adaptation algorithm, we first specify the regression model as
\begin{align}
Y = \Phi^T\theta 
\label{regmodel}
\end{align}
where $\theta$ is the parameter to be estimated; $\Phi$ and $Y$ are the input and output measurements.
Substituting equation (\ref{eqn:kinematicModel1}) into (\ref{eq:vy_dot}), the $a_y$ measurement can be expressed as follows:  
\begin{align}
a_y^{sen}&=\dot{v}_y+v_xr+ g\sin\phi\nonumber\\
&= \frac{-(C_f+C_r)}{mv_x}v_y - \frac{L_fC_f-L_rC_r}{mv_x}r +\frac{C_f}{m}\delta_f.\label{eq:ay_sen}
\end{align}
Then, combining (\ref{eq:ay_sen}) with the yaw rate dynamics, we define the regression model as:
\begin{align}
\label{eq:regression model}
\begin{split}
	\Phi^T&=\begin{bmatrix}
	\frac{-L_f^2r-L_f\hat{v}_{y,k}}{v_{x}}+L_f\delta_{f} &~ \frac{-L_r^2r+L_r\hat{v}_{y,k}}{v_{x}}\\
	\frac{-L_fr-\hat{v}_{y,k}}{v_{x}}+\delta_{f} &~ \frac{L_rr-\hat{v}_{y,k}}{v_{x}}
	\end{bmatrix},\\
	Y&=\begin{bmatrix}
	I_{z}\dot{r} \\
	ma_{y}^{\text{sen}}
	\end{bmatrix},~~~~~ 
	\theta =\begin{bmatrix}
	C_{f} \\
	C_{r}
	\end{bmatrix}
\end{split}
\end{align}
where the unknown lateral velocity is replaced by $\hat{v}_{y,k}$ estimated from the kinematics model.
Observe that all the other time-varying variables in the input and output measurements can be directly obtained from the standard sensors for yaw stability control system. The angular acceleration is obtained by differentiating the yaw rate: $(r[k]-r[k-1])/\Delta t$ with a low-pass filter.
\subsubsection{Adaptation algorithm}
Considering all the input and output data sampled at time instant $i\Delta t$, where $i = 1,2,...k$ is the time step, we want to minimize the sum of the squared prediction errors:
\begin{align}
J(\theta_{k}) = \sum_{i=1}^{k} \lambda^{(k-i)}\left\|Y_i-\Phi_i^T\theta_{k}\right\|_2^2 +\delta \left\|\theta_{k}-{\theta}^+ \right\|_2^2 \label{eq:cost}
\end{align}   
where $0 \ll \lambda < 1$ is the forgetting factor and ${\theta}^+=
[{C}_f^+~ {C}_r^+]^T$ is the nominal values of the front and rear tire cornering stiffnesses. 
Comparing (\ref{eq:cost}) with a standard least square problem, we have included an additional 2-norm regularized term with $\delta > 0$ in order to improve the estimate robustness when the data is less informative or too noisy.   
By setting the partial derivative of $J(\theta_{k})$ with respect to $\theta_{k}$ to zero, the optimal solution, $\theta_{k}^*$, can be derived as follows:
\begin{align}
& \theta_{k}^*=\left(\sum_{i = 1}^{k}\lambda^{k-i}\Phi_i\Phi_i^T+\delta \text{I}_2\right)^{-1}\left(\delta \theta^++\sum_{i=1}^{k}\lambda^{k-i}\Phi_iY_i\right)\nonumber
\end{align}
which implies
\begin{align}
\quad \tilde{\theta}^*_{k} =& \left(\sum_{i = 1}^{k}\lambda^{k-i}\Phi_i\Phi_i^T+\delta \text{I}_2\right)^{-1} \sum_{i=1}^{k}\lambda^{k-i}\Phi_i\tilde{Y}_i \label{eq:optimalSol}\\
\text{where} \quad\quad &\tilde{\theta}_i :=  \theta_i-{\theta}^+,~~~~\tilde{\theta}_{k}^* := \theta_{k}^*-{\theta}^+ \nonumber\\
&\tilde{Y}_i := Y_i-\Phi_i^T{\theta}^+ = \Phi_i^T\tilde{\theta}_i. \nonumber
\end{align}
The expression in (\ref{eq:optimalSol}) is called the batch formulation since it processes the available data set all at once. For simplicity, we can further rewrite the solution in a recursive way as:
\begin{align}
\tilde{\theta}^*_{k} = \tilde{\theta}^*_{k-1}+(R_{k}&+\delta\text{I}_2)^{-1}\left[\delta(\lambda-1)\tilde{\theta}^*_{k-1}+\Phi_{k}e_{k}\right]. \label{eq:recursive adapt law}\\
\text{where~~~~} R_{k} = \sum_{i=1}^{k}&\lambda^{k-i}\Phi_i\Phi_i^T = \lambda R_{k-1}+\Phi_{k}\Phi_{k}^T \label{eq:adaptation gain}\\
e_{k} = \tilde{Y}_{k}&-\Phi^T_{k}\tilde{\theta}_{k-1}^*\nonumber
\end{align}
and $k = 1, 2,...,\infty$. Notice that $(R_{k+1}+\delta\text{I}_2)^{-1}$ in (\ref{eq:recursive adapt law}) is a simple  2-by-2 matrix inversion and the existence of the solution is guaranteed by the regularization term. The recursive formula in (\ref{eq:adaptation gain}) of the adaptation gain $R_{k}$ will help us better understand the stability properties of the adaptation algorithm \cite{landau1990system}.
More details will be discussed in Section 5. In the next section, we summarize the new proposed algorithm for the sideslip angle estimation. 
\subsection{Proposed sideslip angle estimation algorithm}
We have presented the discrete-time dynamics and kinematics observer models in (\ref{eq:dyn discrete-time model}) and (\ref{eq:kin discrete-time model}), respectively.
At each time step, both of the estimated states will be updated using the extended Kalman filters \cite{song1992extended}. The dynamics model is used for estimating the sideslip angle and the kinematics model is used for estimating the tire cornering stiffnesses by applying the adaptation law (\ref{eq:recursive adapt law}). 
Notice that we will enable the adaptation process only when the absolute value of the yaw rate is greater than a certain threshold, $r_t$, in order to have a valid estimated $\hat{v}_{y,k}$ from the kinematics model.
The pseudo code of the estimation algorithm is provided in Algorithm 1.
\begin{figure}[t]
	\begin{center}
		\includegraphics[width = 2.8in]{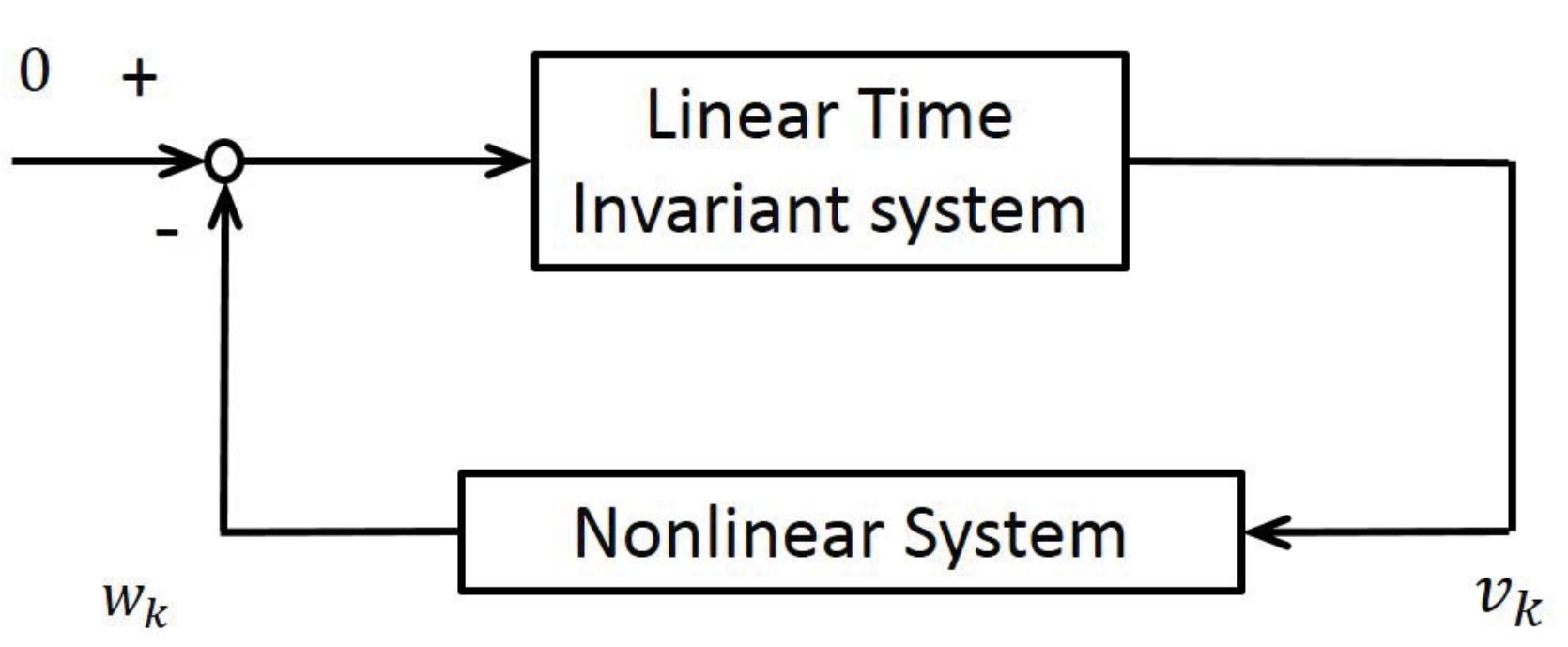}
		\caption{Nonlinear feedback system.} \label{Fig:HyperStabilityBlock}
	\end{center}
\end{figure}
\section{Stability and convergence analysis}
In this section, we study the stability of the proposed adaptation algorithm. In other words, we want to study the convergence property of the tire cornering stiffness estimation error. The energy-based hyperstability approach \cite{anderson1968simplified} is used as it addresses the problem nonlinearity. The analysis will follow three steps.
\begin{enumerate}
	\item[A.] Formulate the adaptation algorithm (\ref{eq:recursive adapt law}) as a nonlinear feedback systems shown in Fig. \ref{Fig:HyperStabilityBlock} where $v_k$ will be linked to the error between the measured output $Y_k$ listed in (\ref{eq:regression model}) at time $k$ and the one predicted according to the adaptation law $\Phi_k^T\theta^*_k$.
	\item[B.] Prove that $v_k$ converges to $0$ for $k\rightarrow\infty$ by using the sufficient conditions of the hyperstability listed in Appendix A. 
	\item[C.] Use the previous results to prove the convergence of the estimated tire cornering stiffnesses to a neighbourhood of the true ones.
\end{enumerate}  
\subsection{Nonlinear feedback formulation of the adaptation algorithm}
We start by formulating the adaptation algorithm (\ref{eq:recursive adapt law}) as a nonlinear feedback system. We first derive a recursive law for the adaptation gain, $(R_{k}+\delta\text{I}_2)^{-1}$. Define $F_{k} = (R_{k}+\delta\text{I}_2)^{-1}$ and substitute $R_k$ as defined in (\ref{eq:adaptation gain}). We obtain: 
\begin{align}
F_{k} &= [\lambda F_{k-1}^{-1}+\delta(1-\lambda)\text{I}_2+\Phi_{k}\Phi_{k}^T]^{-1}. \label{eq:22}
\end{align}
Then, denote $(\sigma_{1,k},\bm{u}_{2,k})$, $(\sigma_{2,k},\bm{u}_{2,k})$ as the first and the second pairs of the singular value and the left singular vector of the matrix $\Phi_{k}$. We can rewrite $\Phi_{k}\Phi_{k}^T$ as
\begin{align}
\Phi_{k}\Phi_{k}^T & = \sum_{j = 1}^{2} \sigma_{j,k}^2\bm{u}_{j,k}\bm{u}_{j,k}^T = \sum_{j=1}^{2} \bm{\phi}_{j,k}\bm{\phi}_{j,k}^T\nonumber
\end{align}
and further represent $\delta(1-\lambda)\text{I}_2+\Phi_{k}\Phi_{k}^T$ as:
\begin{align}
&\delta(1-\lambda)\text{I}_2+\Phi_{k}\Phi_{k}^T = \sum_{j=1}^{2} \mu_{j,k}\bm{\phi}_{j,k}\bm{\phi}_{j,k}^T\label{eq:23}\\
\text{with} \quad&\quad \mu_{j,k} = \frac{\sigma_{j,k}^2+\delta(1-\lambda)}{\sigma_{j,k}^2} \geq 1 \quad \text{for }j = 1, 2.\nonumber
\end{align}
Here, $\mu_{j,k}$ is guaranteed to be finite and always exist because $\sigma_{j,k} >0$. From the expression of $\Phi^T_{k}$ in (\ref{eq:regression model}), we observe that  $\Phi^T_{k}$ is always full rank with exception of singular cases which can be easily discarded in real applications.   
Combing (\ref{eq:22}) and (\ref{eq:23}), we obtain a measurement updated law of the adaptation gain by applying the matrix inverse lemma:  
\begin{align}
F_{k} &= [\lambda F_{k-1}^{-1}+\mu_{1,k}\bm{\phi}_{1,k}\bm{\phi}_{1,k}^T+\mu_{2,k}\bm{\phi}_{2,k}\bm{\phi}_{2,k}^T]^{-1}\nonumber\\
& = F_{k}' - \frac{F_{k}'\bm{\phi}_{2,k}\bm{\phi}_{2,k}^TF_{k}'}{\mu_{2,k}^{-1}+\bm{\phi}_{2,k}^TF_{k}'\bm{\phi}_{2,k}} \label{eq: adapt gain update law 1}\\
\text{where} \quad F_{k}' & = \frac{1}{\lambda}\left(F_{k-1}-\frac{F_{k-1}\bm{\phi}_{1,k}\bm{\phi}_{1,k}^TF_{k-1}}{\lambda\mu_{1,k}^{-1}+\bm{\phi}_{1,k}^TF_{k-1}\bm{\phi}_{1,k}}\right). \label{eq: adapt gain update law 2}
\end{align}
Notice that the updated law of the adaptation gain (\ref{eq: adapt gain update law 1})-(\ref{eq: adapt gain update law 2}) contains two parts. First, $F_{k-1}$ is updated with the first singular vector of the input measurement data, $\phi_{1,k}$ to yield $F_{k}'$. Then $F_{k}'$ is updated based on the second singular vector, $\phi_{2,k}$. 
For this reason, the original sampling time  $k=1,..., T$ is now converted into $n=1,..., 2T$, where $k=\ceil{\frac{n}{2}}$. This will allow us to use the hyperstability theorem which is formulated for SISO systems (Appendix A). Substitute (\ref{eq: adapt gain update law 1}) into (\ref{eq:recursive adapt law}) to obtain:
\begin{align}
\tilde{\theta}^*_{n} &=  \tilde{\theta}^*_{n-1} -\beta_{n}\bm{f}_{n}\bm{\phi}_{n}\bm{\phi}_{n}^T\tilde{\theta}^*_{n-1}+\bm{f}_{n}\bm{\phi}_{n}\tilde{y}_{n}\nonumber\\
&= \tilde{\theta}^*_{n-1} + \bm{f}_{n}\bm{\phi}_{n}(\tilde {y}_{n}-\beta_{n}\bm{\phi}_{n}^T\tilde{\theta}_{n-1}^*)\label{eq:single update law}
\end{align}
\begin{align}
\text{where } \quad \bm{f}_{n}^{-1} &=\alpha_{n}\bm{f}_{n-1}^{-1}+\beta_{n}\phi^T_{n}\phi_{n},~~~~\bm{f}_{0}^{-1} = 0\nonumber\\
\bm{f}_{n} &= \frac{1}{\alpha_{n}}\left(\bm{f}_{n-1}-\frac{\bm{f}_{n-1}\bm{\phi}_{n}\bm{\phi}_{n}^T\bm{f_{n-1}}}{\alpha_{n}\beta_{n}^{-1}+\bm{\phi}_{n}^T\bm{f}_{n-1}\bm{\phi}_{n}}\right)\label{eq:sigle adaptation gain}\\
\tilde{y}_{n} &= \phi_{n}^T\tilde{\theta}_{n},~~~~~\tilde{\theta}_{0}^*=0,~~~~~\alpha_1 = 1. \nonumber\\
\begin{split}
\bm{\phi}_{n} &= \begin{cases}
\bm{\phi}_{1,\ceil{\frac{n}{2}}},&\\
\bm{\phi}_{2,\ceil{\frac{n}{2}}},&
\end{cases}~~~~~
\beta_{n} = \begin{cases}
\mu_{1,\ceil{\frac{n}{2}}},&\\
\mu_{2,\ceil{\frac{n}{2}}},&
\end{cases}\\
\alpha_{n} &= \begin{cases}
\lambda, & \quad \text{if } n \text{ is odd}\\
1, & \quad \text{if } n \text{ is even.}
\end{cases}\label{eq:alphaNbetaDef}
\end{split}
\end{align}
Notice that $\tilde{y}_{n}$ is the measured output and $\bm{\phi}_{n}^T\tilde{\theta}_{n-1}^*$ is the predicted one according to (\ref{eq:single update law}). We then define
\begin{align}
\varepsilon_{n} &= \tilde{y}_{n}-\beta_{n}\bm{\phi}_{n}^T\tilde{\theta}_{n}^*\label{a-posteriori error}\\
\varepsilon^o_{n} &=  \tilde{y}_{n}-\beta_{n}\bm{\phi}_{n}^T\tilde{\theta}_{n-1}^*\nonumber
\end{align}
\begin{figure}[h]
	\centering
	\includegraphics[width = 3.5 in]{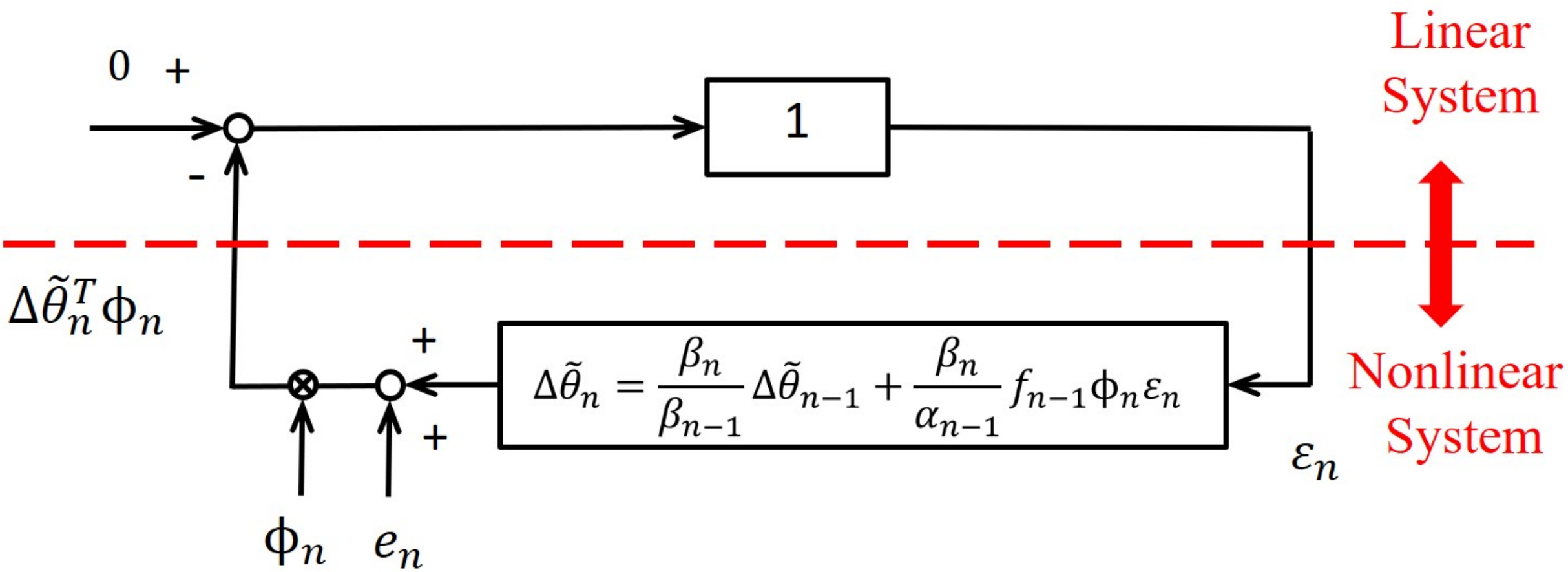}
	\caption{Block diagram of the adaptation algorithm for stability analysis.} \label{Fig:PAAblock}
\end{figure}
as a ``\textit{scaled}" a-posteriori and a ``\textit{scaled}" a-priori predicted measurement errors, respectively. Right multiplying $\bm{\phi}_{n+1}$ to $\bm{f}_{n}$ in equation (\ref{eq:sigle adaptation gain}), we obtain 
\begin{align}
\bm{f}_{n}\phi_{n} = \frac{\bm{f}_{n-1}\bm{\phi}_{n}}{\alpha_{n}+\beta_{n}\bm{\phi}_{n}^T\bm{f}_{n-1}\bm{\phi}_{n}}
\end{align}
and the adaptation law (\ref{eq:single update law}) becomes 
\begin{align}
\tilde{\theta}^*_{n}= \tilde{\theta}^*_{n-1} + \frac{\bm{f}_{n-1}\bm{\phi}_{n}}{\alpha_{n}+\beta_{n}\bm{\phi}_{n}^T\bm{f}_{n-1}\bm{\phi}_{n}}\varepsilon^o_{n}.\label{eq:31}
\end{align}
Then, again left multiplying $-\beta_{n}\bm{\phi}_{n}^T$ to (\ref{eq:31}) and adding $\tilde{y}_{n}$ to both sides of the equation lead to:
\begin{align}
\varepsilon_{n} =  \frac{\alpha_{n}}{\alpha_{n}+\beta_{n}\bm{\phi}_{n}^T\bm{f}_{n-1}\bm{\phi}_{n}}\varepsilon^o_{n}. \label{eq:epsilonRelation}
\end{align}
With this relation, we can express the adaptation law (\ref{eq:31}) using the a-posteriori predicted measurement error $\varepsilon_{n}$ as follows:
\begin{align}
\tilde{\theta}^*_{n} = \tilde{\theta}^*_{n-1}+\frac{1}{\alpha_{n}}\bm{f}_{n-1}\bm{\phi}_{n}\varepsilon_{n}. \label{eq:33}
\end{align}
Define the ``scaled" parameter estimation error as
\begin{align}
\Delta\tilde{\theta}_{n} = \beta_{n}\tilde{\theta}^*_{n}-\tilde{\theta}_{n},
\label{eq:34}
\end{align}
with $\tilde{\theta}_{0} = 0$ and $\beta_0 = 1$. We can rewrite (\ref{a-posteriori error}) and (\ref{eq:33})-(\ref{eq:34}) into the following error dynamics:
\begin{align}
\begin{cases}
\varepsilon_{n} = \phi^T_{n}\tilde{\theta}_{n}-\beta_{n}\bm{\phi}_{n}^T\tilde{\theta}_{n}^*=-\phi^T_{n}\Delta\tilde{\theta}_{n}\\
\Delta\tilde{\theta}_{n} = \frac{\beta_{n}}{\beta_{n-1}}\Delta\tilde{\theta}_{n-1}+\frac{\beta_{n}}{\alpha_{n-1}}\bm{f}_{n-1}\phi_{n}\varepsilon_{n}+e_n
\end{cases} \label{eq:35}
\end{align}
where $e_n = \frac{\beta_{n}}{\beta_{n-1}}\tilde{ \theta}_{n-1}-\tilde{\theta}_{n}$ is treated as an external bounded disturbance. 
Finally, we can represent this error dynamics into the block diagram of Fig. \ref{Fig:PAAblock} which is equivalent to the nonlinear feedback system as shown in Fig. \ref{Fig:HyperStabilityBlock}. 
\subsection{Hyperstability analysis}
\begin{theorem}
	The nonlinear feedback system depicted in the block diagram of Fig. \ref{Fig:PAAblock} with the error dynamics described in (\ref{eq:35}) without the external disturbance term $e_n$ is asymptotically hyperstable (i.e. $\varepsilon_n \rightarrow 0$) if
	\begin{align} \frac{2-\alpha_{n+1}}{\beta_{n}}-\frac{1}{\beta_{n-1}} \geq 0 \quad \forall n = 1,2,...\infty. \label{eq:eta condition}
	\end{align}
\end{theorem}
\begin{figure}[t]
	\centering
	\includegraphics[width = 3.7 in]{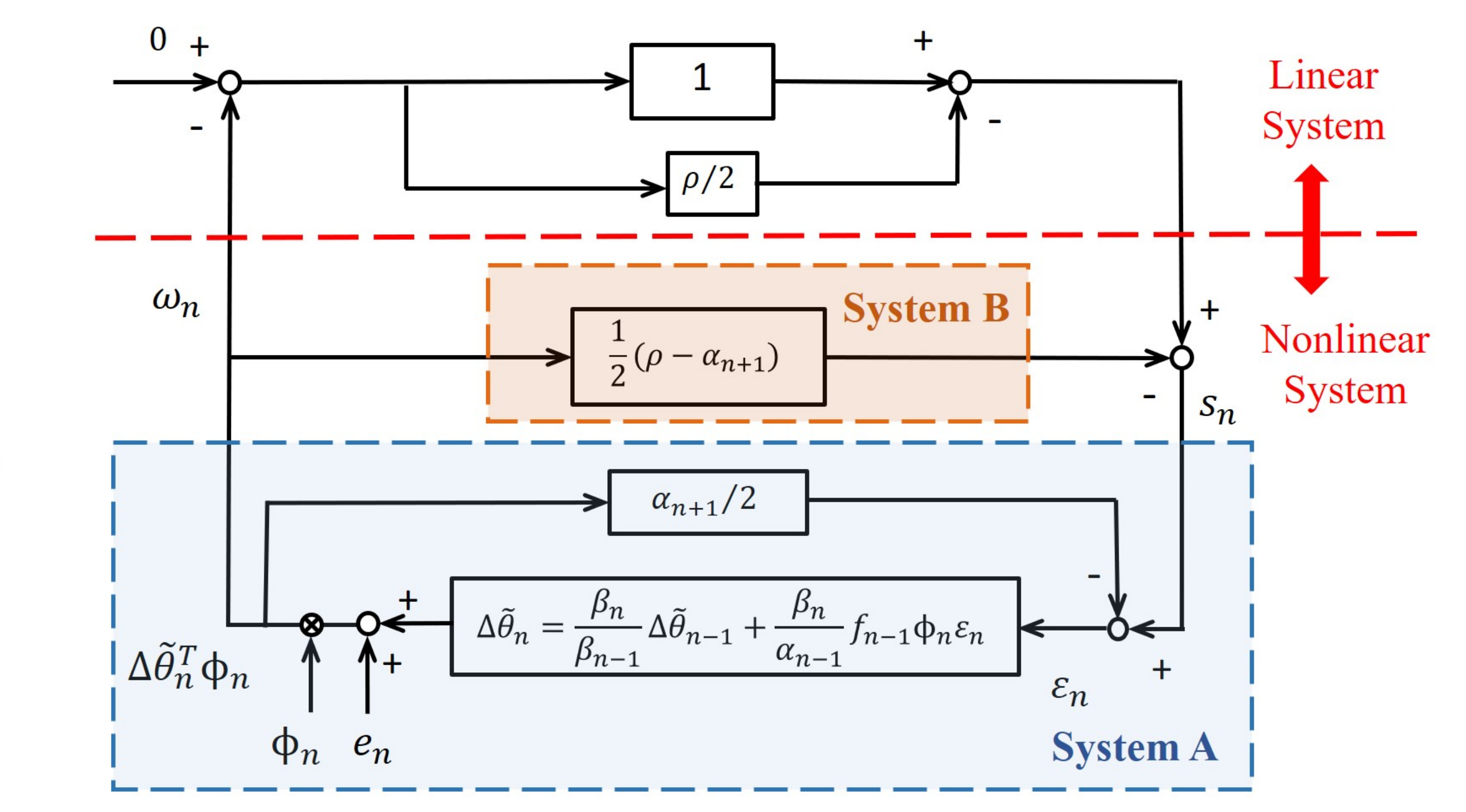}
	\caption{Equivalent system of the block diagram in Fig. \ref{Fig:PAAblock} for stability analysis.} \label{Fig:PAAblock1}
\end{figure}
\begin{proof}
	To prove the theorem we use the concept of hyperstability \cite{anderson1968simplified} briefly summarized in Appendix A.
	Next, we prove that the sufficient conditions listed in Theorem A.1 are satisfied.
	First, we notice that the forward linear system is the identity, which is obviously strictly positive real. However, the nonlinear block of the adaptation algorithm does not satisfy the Popov inequality. Therefore, we further modify the system and rewrite it as the one shown in the block diagram of Fig. \ref{Fig:PAAblock1}. Since the same signals have been added and subtracted in the feedback system, the stability property will not change.
	
	We now again check the sufficient conditions for the modified system. Start with the requirement of the nonlinear feedback block satisfying Popov inequality.
	Considering System A in the block diagram of Fig. \ref{Fig:PAAblock1} with input ($s_{n}$) and output ($w_{n}$) signals, we have
	\begin{align}
	w_{n} &= \Delta \tilde{\theta}^T_{n}\bm{\phi}_{n},~~
	s_{n} = \varepsilon_{n}+\frac{\alpha_{n+1}}{2}\Delta \tilde{\theta}^T_{n}\bm{\phi}_{n}\nonumber\\
	\bm{\phi}_{n}\bm{\phi}_{n}^T &= \beta_{n}^{-1}(\bm{f}_{n}^{-1}-\alpha_{n}\bm{f}_{n-1}^{-1}) \nonumber\\
	\bm{\phi}_{n}\varepsilon_{n} &=\frac{\alpha_{n}}{\beta_{n}}\bm{f}_{n-1}^{-1}\Delta\tilde{\theta}_{n} -\frac{\alpha_{n}}{\beta_{n-1}}\bm{f}_{n-1}^{-1}\Delta\tilde{\theta}_{n-1} \label{phiError}
	\end{align}
	from equations (\ref{eq:sigle adaptation gain}) and (\ref{eq:35}) without considering the external disturbance$e_n$. Define
	\begin{align}
	\eta_n = \frac{2-\alpha_{n+1}}{\beta_{n}}-\frac{1}{\beta_{n-1}} \nonumber   
	\end{align}
	for the sake of simplicity in later expression. The sum of the product of $w_n$ and $s_n$ can be calculated as:
	\begin{align}
	~~~&\sum_{n=1}^{2k} w_{n}s_{n} = \sum_{n=1}^{2k} \Delta \tilde{\theta}^T_{n}\bm{\phi}_{n}(\varepsilon_{n}+\frac{\alpha_{n+1}}{2}\Delta \tilde{\theta}^T_{n}\bm{\phi}_{n})\nonumber\\
	=&\sum_{n=1}^{2k}\frac{\alpha_{n}}{2}\frac{1}{\beta_{n-1}}(\Delta\tilde{\theta}_{n}^T-\Delta\tilde{\theta}_{n-1}^T)\bm{f}_{n-1}^{-1}(\Delta\tilde{\theta}_{n}-\Delta\tilde{\theta}_{n-1})\nonumber\\
	&~~+\sum_{n=1}^{2k}\frac{\alpha_{n+1}}{2\beta_{n}}\Delta \tilde{\theta}^T_{n}\bm{f}_{n}^{-1}\Delta \tilde{\theta}_{n}-\frac{\alpha_{n}}{2\beta_{n-1}}\Delta \tilde{\theta}^T_{n-1}\bm{f}_{n-1}^{-1}\Delta \tilde{\theta}_{n-1}\nonumber\\
	&~~~~~~~~+\sum_{n=0}^{2k}\frac{\eta_n\alpha_{n}}{2}\Delta\tilde{\theta}_{n}^T\bm{f}_{n-1}^{-1}\Delta\tilde{\theta}_{n}\nonumber\\
	=&\sum_{n=1}^{2k}\frac{\alpha_{n}}{2\beta_{n-1}}(\Delta\tilde{\theta}_{n}^T-\Delta\tilde{\theta}_{n-1}^T)\bm{f}_{n-1}^{-1}(\Delta\tilde{\theta}_{n}-\Delta\tilde{\theta}_{n-1})\nonumber\\
	&~~~~~~~~+\frac{\alpha_{2k+1}}{2\beta_{2k}}\Delta \tilde{\theta}^T_{2k}\bm{f}_{2k}^{-1}\Delta \tilde{\theta}_{2k}-\frac{\alpha_{1}}{2\beta_{0}}\Delta \tilde{\theta}^T_{0}\bm{f}_{0}^{-1}\Delta \tilde{\theta}_{0}\nonumber\\
	&~~~~~~~~~~~~+\sum_{n=1}^{2k}\frac{\eta_n\alpha_{n}}{2}\Delta\tilde{\theta}_{n}^T\bm{f}_{n-1}^{-1}\Delta\tilde{\theta}_{n}.\label{eq:Pupov inequality}
	\end{align}
	Since all the variables of $\alpha_n$, $\beta_n$ are positive as defined in (\ref{eq:alphaNbetaDef}), we can know that the sum of the product of $w_n$ and $s_n$ will have a lower bound 
	\begin{table}[t]
	\caption{Model Parameters}
	\centering
	{\begin{tabular}{|c|c|c|c|}
			\hline
			vehicle mass & $m$ & 2300.132 & kg$~~$\\
			\hline
			vehicle rotational inertia & $I_z$ & 4400 & kgm$^2$\\
			\hline
			distance from COG to front axle & $L_f$ & 1.505 & m$~~$\\
			\hline
			distance from COG to rear axle & $L_r$ & 1.504 & m$~~$\\
			\hline
			front tire conering stiffness & $C_f$ & 160776 & N/rad$^2$\\
			\hline
			rear tire conering stiffness & $C_r$ & 254100 & N/rad$^2$\\
			\hline
			gravity & $g$ & 9.80665 & m/s$^2$\\
			\hline
	\end{tabular}}
	\label{table1}
\end{table}
	\begin{align}
	    \sum_{n=1}^{2k} w_{n}s_{n}\geq& -\frac{\alpha_{1}}{2\beta_0}\Delta \tilde{\theta}^T_{0}\bm{f}_{0}^{-1}\Delta \tilde{\theta}_{0} = 0 \nonumber 
	\end{align}
	and satisfy the Pupov inequality with a condition of
	\begin{align}
	\eta_n \geq 0 \quad \forall n = 1,2,...,2k.
	\end{align}
	Next, considering the time varying linear System B in the block diagram of Fig. \ref{Fig:PAAblock1}, we find that it also satisfies the Popov inequality since
	\begin{align}
	\frac{1}{2}(\rho - \alpha_{n+1}) \geq 0 \quad \text{for choosing }~ 1 \leq \rho<2. 
	\end{align}
	Then, the overall nonlinear feedback system as shown in the block diagram of Fig. \ref{Fig:PAAblock1} satisfies the Popov inequality since it is made by a feedback connection of two passive systems, A and B. Finally, the linear feedforward system for the modified system, $1-\rho/2$, is strictly positive real for having $1\leq \rho < 2$. Now, we know that the adaptation system is hyperstable. In other words, $|(1-\rho/2)w_{n}| < \infty$ is bounded. This will further imply that the output of the nonlinear feedback system, $w_n < \infty$, is bounded as well. Therefore, having all of three requirements listed in Appendix A, we can conclude that the adaptation system without the external disturbance is asymptotic hyperstable $\varepsilon_n \rightarrow 0$. 
\end{proof}
We start from the analysis in Theorem V.1 and consider the effect of the external disturbance $e_n$. 
\begin{theorem}
	Assume that the two norm of the difference of the true parameter $\tilde{\theta}_k$ between two consecutive steps is bounded:
	\begin{align}
	\|\tilde{\theta}_{k-1}-\tilde{\theta}_{k}\|_2 \le \Lambda ~~~~\quad \forall k = 1,2,...,\infty. \nonumber
	\end{align}
	Consider the regularized weighted least square problem (\ref{eq:cost}) with the regression model described in (\ref{eq:regression model}). 
	Then, there exists a set of parameters $\delta>0, 0 \ll \lambda < 1$  satisfying the condition:
	\begin{align} \sigma_{1,k}^2\sigma_{2,k}^2 +\delta(2-\lambda)\sigma_{2,k}^2-\delta\sigma_{1,k}^2 \geq 0 ~~ \forall k = 1,2,...,\infty \label{eq:PAAcondition}
	\end{align}
	which guarantees $\varepsilon_{2k}\in R(\Lambda)$ for $R(\Lambda)$ being a ball of radius $\Lambda$ centred in the origin.
\end{theorem}
\begin{proof}
We use the same steps as in Theorem V.1. Consider the input signal $\varepsilon_{n}$ of System A with 
\begin{align}
	\phi_{n}\varepsilon_{n} =\frac{\alpha_{n}}{\beta_{n}}\bm{f}_{n-1}^{-1}(\Delta\tilde{\theta}_{n}+&\tilde{\theta}_{n})-\frac{\alpha_{n}}{\beta_{n-1}}\bm{f}_{n-1}^{-1}(\Delta\tilde{\theta}_{n-1}+\tilde{\theta}_{n-1})\nonumber
\end{align}
deriving from (\ref{eq:35}). We can derive the same Popov inequality as shown in (\ref{eq:Pupov inequality}) for System A but with an extra term of
\begin{align}
	\sum_{n=1}^{2k}\alpha_{n}\Delta\tilde{\theta}_{n}^T\bm{f}_{n-1}^{-1}(\frac{1}{\beta_{n}}\tilde{\theta}_{n}-\frac{1}{\beta_{n-1}}\tilde{\theta}_{n-1}).\label{eq:extra term}
\end{align}
Therefore, the same requirement of $\eta_n = \frac{2-\alpha_{n+1}}{\beta_{n}}-\frac{1}{\beta_{n-1}} \geq 0$ in Theorem V.1 is necessary for hyperstability. From this, we can easily infer the condition of (\ref{eq:PAAcondition}) by expanding out $\beta_{n-1}$, $\beta_{n}$ and $\alpha_{n+1}$ using  (\ref{eq:alphaNbetaDef}).
\begin{table}[t]
	\caption{Estimator Parameters (Algorithm 1)}
	\centering
	{\begin{tabular}{|c|c|c|}
			\hline
			sampling time & $\Delta t$ & 0.01 sec$~~$\\
			\hline
			covariance matrix of $w_k[k]$ & $W_k$ & \text{diag}$([0.2,0.6,0.05])$\\
			\hline
			covariance matrix of $v_k[k]$ & $V_k$ & 0.05\\
			\hline
			covariance matrix of $w_d[k]$ & $W_d$ & \text{diag}$([6,0.5,0.1,0.0002])$\\
			\hline
			covariance matrix of $v_d[k]$ & $V_d$ & \text{diag}$([0.1,0.01])$\\
			\hline
			forgetting factor & $\lambda$ & 0.975\\
			\hline
			regularized term weighting& $\delta$ & 0.02\\
			\hline
			yaw rate threshold & $r_{t}$ & 0.1~$\text{rad}/\text{s}$\\
			\hline
	\end{tabular}}
	\label{table2}
\end{table}
Next, combining the first term in the right hand side of (\ref{eq:Pupov inequality}) together with (\ref{eq:extra term}), we can conclude that System A will satisfy Pupov inequality under the conditions:
\begin{align}
	\frac{\eta_n}{2}\Delta\tilde{\theta}_{n}^T\bm{f}_{n-1}^{-1} \Delta\tilde{\theta}_{n}+\Delta\tilde{\theta}_{n}^T\bm{f}_{n-1}^{-1}&(\frac{1}{\beta_{n}}\tilde{\theta}_{n}-\frac{1}{\beta_{n-1}}\tilde{\theta}_{n-1}) \geq 0 \nonumber\\
	&\quad\quad\forall n = 1,2,...,2k\nonumber\\
	\implies
	\|\Delta\tilde{\theta}_{n}\|_2 \geq \frac{2\kappa(\bm{f}_{n-1}^{-1})}{\eta_n}&\left\|\frac{1}{\beta_{n}}\tilde{\theta}_{n}-\frac{1}{\beta_{n-1}}\tilde{\theta}_{n-1}\right\|_2 \nonumber\\
	&\quad\quad\forall n = 1,2,...,2k \label{eq:thetaUpperbound}
\end{align}
where $\kappa(\cdot)$ denotes the condition number of the positive definite matrix $\bm{f}_{n-1}^{-1}$. Since the rate of $\tilde{\theta}$ is bounded by the assumption, the existence of the right hand side in (\ref{eq:thetaUpperbound}) is guaranteed. Then, based on passivity theorem \cite{khalil1996noninear}, we know that there exists a time-varying energy function which is positive definite and is dissipating over time in the region of
\begin{align}
	\|\Delta\tilde{\theta}\|_2 \geq \max_n  \frac{2\kappa(\bm{f}_{n-1}^{-1})}{\eta_n}\left\|\frac{1}{\beta_{n}}\tilde{\theta}_{n}-\frac{1}{\beta_{n-1}}\tilde{\theta}_{n-1}\right\|_2 .\label{eq:conditionDeltaTheta}
\end{align}
Equation (\ref{eq:conditionDeltaTheta})
implies the boundedness of the predicted measurement error $\epsilon_n$ in the adaptation algorithm. 
\end{proof}
\noindent\textit{\textbf{Remark:}} Some considerations can be drawn from the analysis in Theorem V.1 and Theorem V.2.
First, the boundedness of the adaptation error depends on the time-varying rate of change of $\tilde{\theta}$. According to the result shown in (\ref{eq:conditionDeltaTheta}), a larger value in $\left\|\frac{1}{\beta_{n}}\tilde{\theta}_{n}-\frac{1}{\beta_{n-1}}\tilde{\theta}_{n-1}\right\|_2$ will lead to a larger bound in the predicted measurement error. Therefore, we can expect a better adaptation performance under non-extreme driving scenarios. Second, the boundedness of the adaptation error shrinks as the condition number of $\bm{f}_{n-1}^{-1}$ decreases. This highlights the importance of input measurement matrices, $\Phi^T_i$, being well-conditioned in the regression model. By observing (\ref{eq:regression model}), we can infer that the conditional number of $\Phi^T$ is roughly equal to the ratio of it's (2, 1) and (2, 2) components, since $L_f \approx L_r$ for a vehicle. In order to avoid a bad adaptation performance, we should add an additional condition of
\begin{figure}[t]
	\begin{center}
		\includegraphics[width = 3.5in]{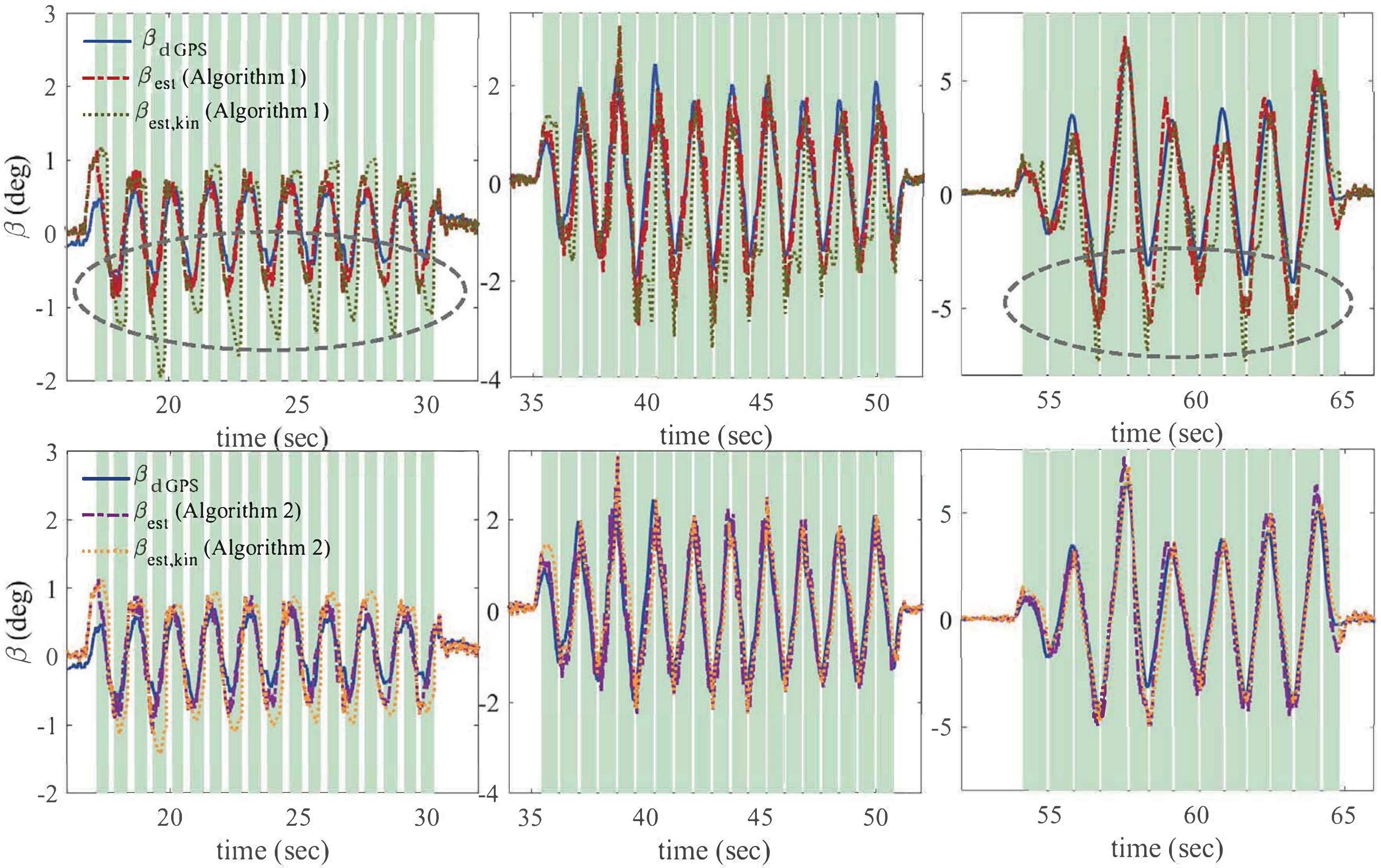}
		\caption{Performance comparison of Algorithm 1 and Algorithm 2 for a slalom test.} 
		\label{fig. a1}
	\end{center}
\end{figure}
\begin{align}
\frac{1}{c_t} \le \left|\frac{\Phi^T(2,1)}{\Phi^T(2,2)}\right| \le c_t \label{eq:conditional number}
\end{align}       
to enable the adaptation process in our proposed sideslip angle algorithm where $c_t > 0$ is the maximum allowed conditional number of measurement data. Third, according to (\ref{eq:Pupov inequality}), $\eta_n$ is the energy dissipation rate of the system. Therefore, a larger value of $\eta_n$ results in a faster convergence rate. Based on an further analysis in the condition of $\eta_n \geq 0$, we can obtain a good starting value of the regularization weight
\begin{align}
\delta \approx 1 \left/\right. (\frac{1}{\sigma_{2,k}^2}-\frac{1}{\sigma_{1,k}^2})
\end{align}
 This is derived by rewriting (\ref{eq:PAAcondition}) into the following form:
\begin{align}
2-\lambda \geq \sigma_{1,k}^2\left(\frac{1}{\sigma_{2,k}^2}-\frac{1}{\delta}\right) \quad \text{for ~} 0 \ll \lambda < 1.
\end{align} 
\subsection{Convergence of the estimated cornering stiffness}
\begin{theorem}
The asymptotical hyperstability of the nonlinear feedback system depicted in Fig. \ref{Fig:PAAblock} with $e_n=0$ guarantees that the estimated parameters converge to $\frac{1}{\beta_n}\tilde{\theta}_n$
\begin{align}
	\lim_{n \rightarrow \infty}\tilde{\theta}_n^* \rightarrow \frac{1}{\beta_n}\tilde{\theta}_n \nonumber
\end{align}
\end{theorem}
\begin{proof}
	In Theorem V.1, the stability proof shows the convergence of  $\varepsilon_n$ without the disturbance term. 
	\begin{align}
	    \varepsilon_n &= \bm{\phi}^T_{n}(\tilde{\theta}_{n}-\beta_{n}\tilde{\theta}_{n}^*) \rightarrow 0. \label{converge1}
	\end{align}
	Then, recalling from (\ref{eq:epsilonRelation}), we know that $\varepsilon_n^0$ will also converge to zero for a bounded $\bm{\phi}_n$. We have 
	\begin{align}
	\varepsilon_{n+1}^0 &= \bm{\phi}^T_{n+1}(\tilde{\theta}_{n+1}-\beta_{n+1}\tilde{\theta}_{n}^*)\rightarrow 0 \nonumber 
	\end{align}
	which can be further rewritten as 
	\begin{align}
    \bm{\phi}^T_{n+1}(\tilde{\theta}_{n}-\beta_{n}\tilde{\theta}_{n}^*)\rightarrow 0 \label{eq:a}
	\end{align}
	by substituting $\tilde{\theta}_{n+1} =  \frac{\beta_{n+1}}{\beta_n}\tilde{\theta}_n$ under the assumption of no external disturbance.
	
	Combining the results of (\ref{converge1}) and (\ref{eq:a}) and using the fact  that $\tilde{\theta}_{n}-\beta_{n}\tilde{\theta}_{n}^*$ cannot be orthogonal to $\bm{\phi}_n$ and $\bm{\phi}_{n+1}$ since $\bm{\phi}_n$ and $\bm{\phi}_{n+1}$ span the whole state space, we can conclude that $\tilde{\theta}_{n}-\beta_{n}\tilde{\theta}_{n}^*$ will approach zero as $n \rightarrow \infty$.
\end{proof}
From Theorem V.2 and Theorem V.3 we can conclude that the estimated tire cornering stiffness coefficients will converge to a neighbourhood of the true values when we include the external disturbance $e_n$ term in the proof.
\begin{figure}[t]
	\begin{center}
		\includegraphics[width = 3.5in]{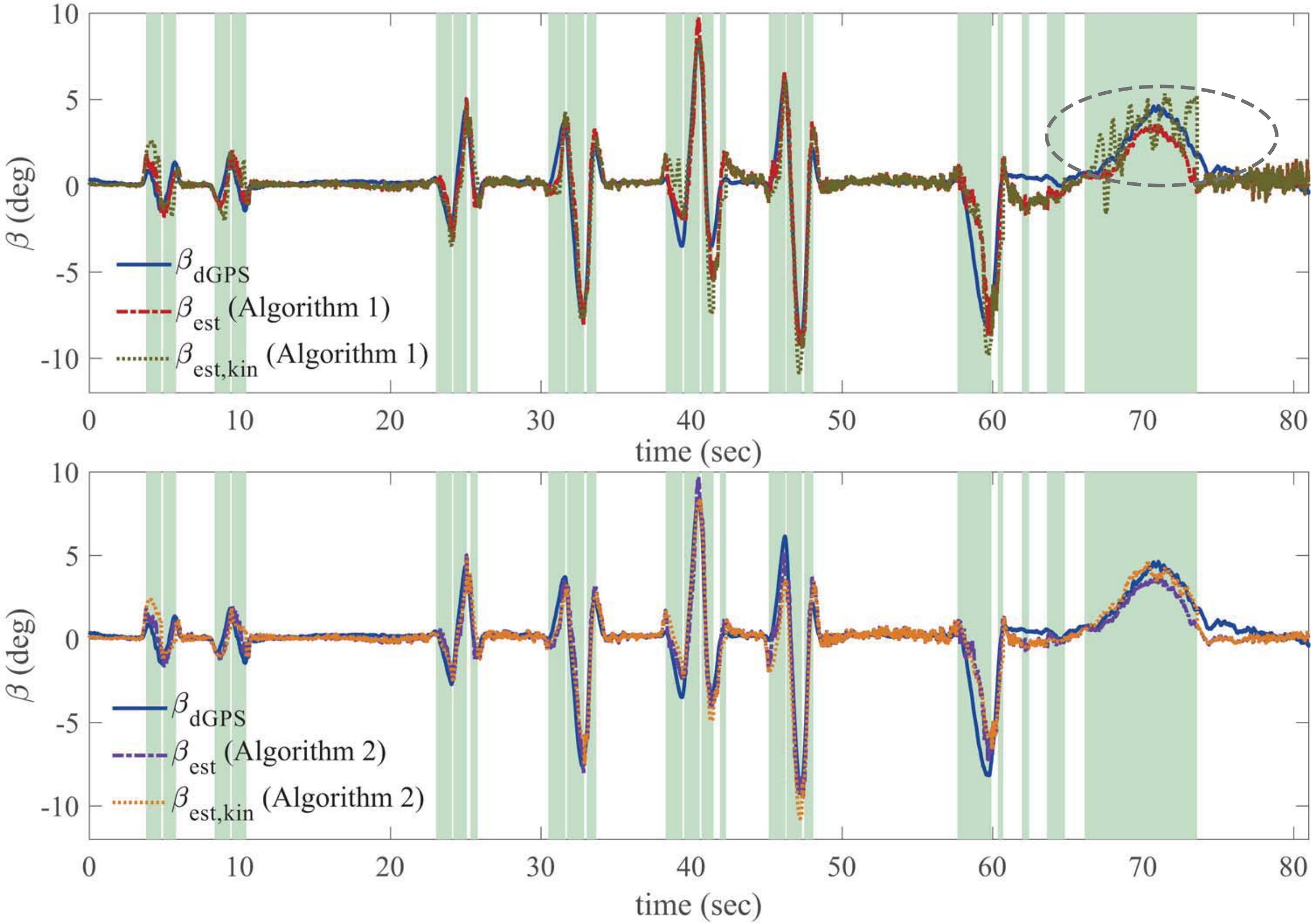}
		\caption{Performance comparison of Algorithm 1 and Algorithm 2 for a severe single lane changing.} 
		\label{fig. a2}
	\end{center}
\end{figure}
\section{Algorithm Improvement}
In this section, the proposed estimation algorithm is evaluated with real experimental tests. We first conducted a slalom and a severe single lane changing tests at Hyundai-Kia Motors California Proving Ground. The first test setting consists of eleven lined up cones, separated by 18 m. The vehicle is driven through the course in a slalom pattern at constant speed, 50 km/hr. The second one is a standardized maneuver which generates a peak lateral acceleration of approximately 0.6g. A further analysis and a small modification of Algorithm 1 are provided based on the estimation results.
\subsection{Experimental Setup}
Our experimental vehicle is a $5^{th}$ generation Hyundai Genesis equipped with a differential global positioning system (dGPS) Oxford TR3000. A real-time kinematics (RTK) technology is adopted to allow an accuracy down to 2-4 cm for position measurement. We will consider the measured sideslip angle provided from dGSP as a ground truth to validate the estimated performance. The real–time computations are performed on a dSPACE DS1401 Autobox system which consists of a IBM PowerPC 750GL processor running at 900 MHz. The aforementioned hardware components communicate through a CAN bus and the estimation algorithm is executed at 100 Hz.

Table \ref{table1} shows the nominal model parameters of the test vehicle and Table \ref{table2} shows the estimation parameters for Algorithm 1. We initialize the measurement noise covariances by processing the measurement outputs while they are held constant. Since the values of the noise covariances are all small, we then apply a reasonable scaling factor to avoid the numerical issue before the tuning. The process noise covariance matrix is picked based on the unmodeled dynamics. According to the results shown in Table \ref{table2}, we can see that the process noise covariance of the $v_y$ equation is chosen to be relative bigger than other states since the coupling of the roll dynamics has been ignored and the gravity effect causes more influence on $v_y$ dynamics. Similarly, we choose the process noise covariance of the sensor bias, $d$, to be significantly small because we believe that the offset is ``nearly" constant. In other words, we can treat the dynamics of $d$ as arbitrarily-slowly time-varying. For the forgetting factor, since it determines the rate of change of the weighting factors of the regression errors, we start with the value vary close to 1 for the fact that the tire cornering stiffness varies with the maneuver and our sampling rate is way fast enough to capture its varying speed. Then, we gradually decrease the value to allow more weighting on recent data to improve the performance.
\begin{figure}[t]
	\begin{center}
		\includegraphics[width = 3.2in]{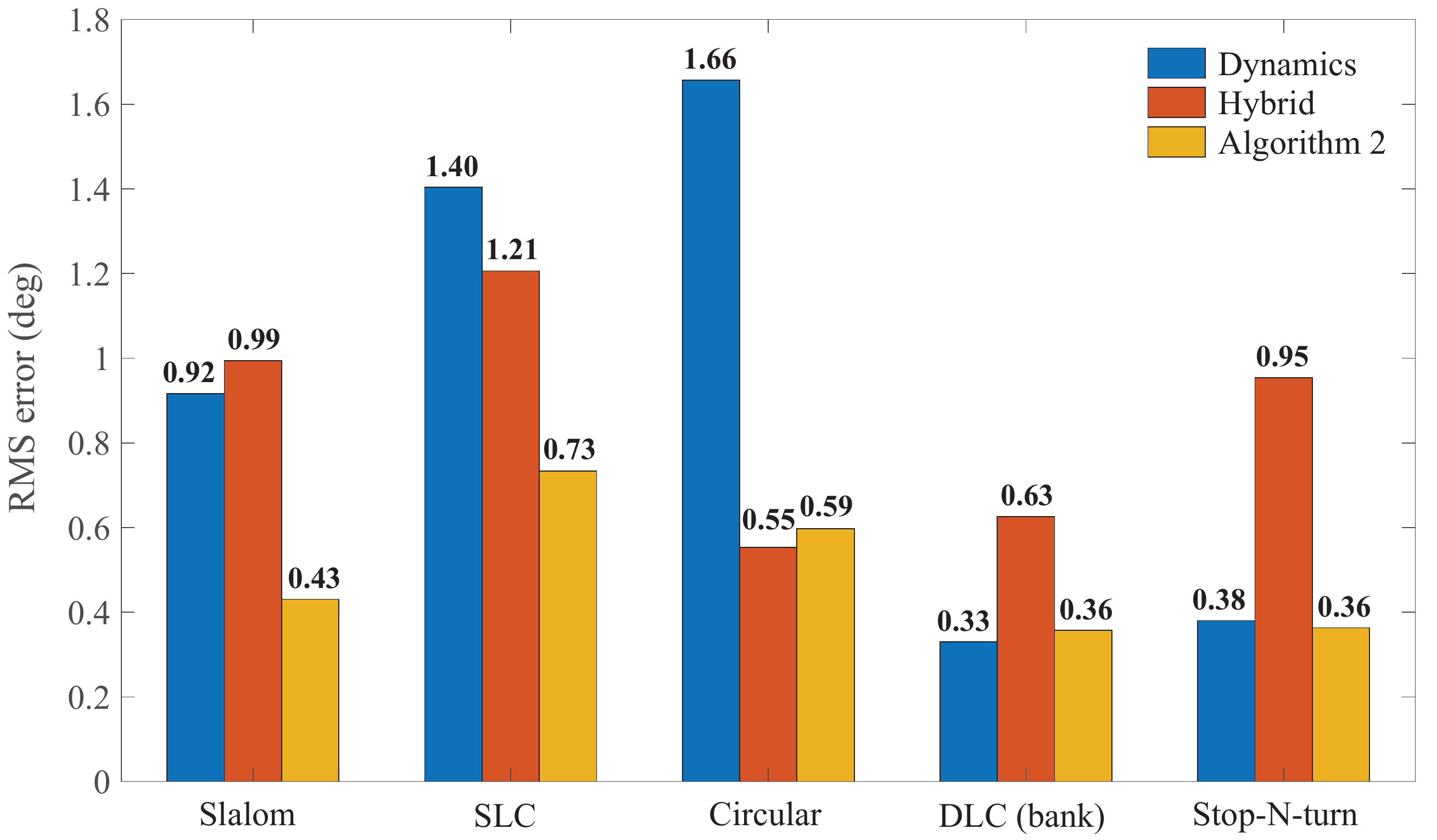}
		\caption{The root mean square errors of the proposed method compared with the existing methods.} 
		\label{Fig:RMS}
	\end{center}
\end{figure}
\subsection{Motivation}
The estimation results of a slalom and a severe single lane changing tests are shown in the upper plots of Fig. \ref{fig. a1} and Fig. \ref{fig. a2} respectively. Note that the light green background represents the condition of $|r|>r_t$ indicating that the adaptation algorithm is active. Comparing with the solid blue ($\beta_{dGPS}$) and the dashdotted red ($\beta_{est}$) lines, we can see that the proposed method performs well. However, there is still room for improvement in the region highlighted with gray dashed lines. In these regions the estimated sideslip angle ($\hat{v}_{y,k}$) provided from the kinematics model (\ref{eq:kin discrete-time model}) is noisy. 
This affects the output measurements $\Phi$ in the regression model which is used for the cornering stiffness adaptation. 
To address this issue, we proposed a small modification for Algorithm 1 which is described next.
\begin{algorithm*}
	\centering
	\caption{Sideslip Angle Estimation}
	\begin{algorithmic}[1]
		\Initialize{$\hat{\mathbf{x}}_{k}[0] \gets [v_x[0]~v_y[0]~0]^T$,~~$\hat{\mathbf{x}}_{d}[0] \gets [v_y[0]~r[0]~0~~0]^T$,~~$\tilde{\theta}^*_0 \gets 0$,~~ $R_0 \gets 0$,~~$\theta_0^* \gets {\theta}^++\tilde{\theta}^*_0$,}
		\State ~~~~$P_k[0] \gets P_{k,0}$,~~$P_d[0] \gets P_{d,0}$~~~~~~~~~~~~~~~~//\textit{ initialize prior means and estimate error covariance matrices for} \text{EKF}
		\For{i = 1 to k+1}
		\State $\hat{\mathbf{x}}_{d}[i] \gets$ EKFupdate($\hat{\mathbf{x}}_{d}[i-1],\mathbf{u}_{d}[i-1],\mathbf{u}_{d}[i],\mathbf{y}_{d}[i],P_d[i-1],\theta^*_{i-1}$)~~~~~~~~~~~// EKF \textit{update with model (\ref{eq:dyn discrete-time model})}
		\State $\hat{\mathbf{x}}_{k}[i] \gets$ EKFupdate($\hat{\mathbf{x}}_{k}[i-1],\mathbf{u}_{k}[i-1],\mathbf{y}_{k}[i],P_k[i-1]$))~~~~~~~~~~~~~~~~~~~~~~~// EKF \textit{update with model (\ref{eq:kine model modi})}
		\State $\mathbf{u}_{k}[i] \gets [a_x[i]~~a_y^{sen}[i]-g\sin\hat{\phi}_d[i]-\hat{d}_d[i]]^T$
		\If{$|r_i| \geq r_t$ \textbf{and} $1/c_t \leq |\Phi_{i}^T(2,1)/\Phi_{i}^T(2,2)| \leq c_t$}
		\State $R_{i} = \lambda R_{i-1}+\Phi_{i}\Phi_{i}^T$~~~~~~~~~~~~~~~~~~~~~~~~~~~//\textit{ obtain the input measurement $\Phi_{i}$ from (\ref{eq:regression model})}
		\State $\theta^*_i \gets {\theta}^++$ AdaptationUpdate($R_i, \tilde{\theta}^*_{i-1}$)~~~~~~//\textit{ apply the recursive update law (\ref{eq:recursive adapt law}) for the tire cornering stiffnesses}
		\Else 
		\State $\theta^*_i \gets \theta^*_{i-1}$ 
		\State $\hat{\mathbf{x}}_{k}[i] \gets [v_x[i]~\hat{v}_{y,d}[i]~\sin\hat{\phi}_d[i]]^T$ ~~~~~~~~~~~~//\textit{ update the state estimates of the} \text{EKF} \textit{for model (\ref{eq:kine model modi})}
		\State $P_k[i] \gets \text{diag}(0,P_d[i](1,1),P_d[i](3,3))$ ~~~~~//\textit{ update error covariance matrix of the} \text{EKF} \textit{for model (\ref{eq:kine model modi})}
		\EndIf        
		\State $\beta[i] \gets \tan^{-1}(\hat{v}_{y,d}[i]/v_x[i])$~~~~~~~~~~~~~~~~~~~~~~~~//\textit{ calculate sideslip angle}
		\EndFor
	\end{algorithmic}
\end{algorithm*}
\subsection{Modification to Algorithm 1}
In this section, we improve Algorithm 1 proposed in the previous section. We will show the performance of the new algorithm. However, the convergence analysis is harder to establish because of the tightly coupling between two observers. 

According to the discussion above, we want to improve the estimation of the kinematics model by considering the bank angle effect. Start by deriving from the lateral dynamics (\ref{eq:vy_dot}) and the lateral acceleration models (\ref{eq:aySen}). We can get the following relation:
\begin{align}
\dot{v}_y = -v_x r+a_y^{sen}-g\sin\phi-d. \label{eq:modify1}
\end{align}
Then, based on the result shown in (\ref{eq:modify1}), the original kinematics model (\ref{eq:kin discrete-time model}) in Algorithm 1 can be modified into:
\begin{align}
\begin{split}
A_k(t) &= \begin{bmatrix}
~0 &~~ r(t)\\
-r(t) &~~ 0
\end{bmatrix},~~~~
B_k(t) = \begin{bmatrix}
1~ &~ 0\\
0~ &~ 1
\end{bmatrix},\\
C_k(t) &= \begin{bmatrix}
1 ~&~ 0
\end{bmatrix}. \label{eq:kine model modi}
\end{split}
\end{align}  
with the estimated state and the input vectors defined as:
\begin{align}
    \hat{\mathbf{x}}_k =\begin{bmatrix} v_x\\v_y
    \end{bmatrix},~~~ \mathbf{u}_k = \begin{bmatrix}a_x\\a_y^{sen}-g\sin\hat{\phi}_{d}-\hat{d}_{d}\end{bmatrix}. \nonumber
\end{align}
We can see that the measured lateral acceleration in $\mathbf{u}_k$ is added with an additional term, $-g\sin\hat{\phi}_{d}-\hat{d}_{d}$, where $\sin{\hat{\phi}_{d}}$ and $\hat{d}_{d}$ are the estimated values from the dynamics model (\ref{eq:dyn discrete-time model}).
With this modification, the kinematics model (\ref{eq:kine model modi}) does not remain unaffected by the vehicle parameters anymore. However, we can claim that the estimated term of  $-g\sin\hat{\phi}_{d}-\hat{d}_{d}$ from the dynamics model is relatively less sensitive to the model error in the normal driving situations for $\dot{v}_{y}$ being small and slowly varying since
\begin{align}
-g\sin\hat{\phi}_{d}-\hat{d}_{d} = \hat{\dot{v}}_{y,d}+v_x\hat{r}-a_y^{sen}.\label{eq:modify2}
\end{align} 
By observing (\ref{eq:modify2}), we can expect that $-g\sin\hat{\phi}_{d}-\hat{d}_{d}$ will mostly depend on the error of $\hat{\dot{v}}_{y,d}$ because $\hat{r}$ and $a_y^{sen}$ are directly relevant to the values measured from the sensors. 
Although it will be bias more when $\dot{v}_{y}$ is large, we still can expect it with a similar trend and without too much difference from the true value.

The modified version of the estimator is provided in Algorithm 2. Moreover, we add an additional condition listed in (\ref{eq:conditional number}) to guarantee well-conditioned measurement data for tire cornering stiffness adaptation. 
The estimated performance of Algorithm 2 is shown in the bottom plots of Fig. \ref{fig. a1} and Fig. \ref{fig. a2}. We can see a significant improvement in the estimated value of $\hat{v}_{y,k}$.
\section{Experimental verification}
Having a modified version of the estimator (Algorithm 2), to evaluate its robustness, three more different tests of severe and normal steering maneuvers under different road conditions are conducted and all the tests are listed as follows:
\begin{enumerate}
    \item[1)] a slalom test on a low friction flat road,
    \item[2)] a severe single lane changing on a normal flat road,
    \item[3)] a steady circular motion test on a normal flat road,
    \item[4)] a double lane changing test on a road with significant bank angle, and
    \item[5)] a stop-N-turn test on a normal road.
\end{enumerate}
To display the advantage of Algorithm 2, we further compare the experimental results with other two methods: 
\begin{enumerate}
    \item[1)] Dynamics observer: a dynamics estimator with a state augmented with bank angle and sensor bias without cornering stiffness adaptation, and
    \item[2)] Hybrid observer: a hybrid estimator switching between the dynamics model and a kinematics model described in Algorithm 2.
\end{enumerate}
All parameters required in Algorithm 2 are the same as Algorithm 1 listed in Table \ref{table2} except that the covariance matrix of $w_k[k]$ is set to be diag$([0.2,0.6])$ and the maximum conditional number, $c_{t}$, is 20.  
\subsection{Experimental results for Algorithm 2}
The experimental results are shown in Fig. \ref{Fig:Slalom}-\ref{Fig:stopnturn2} and the comparison of RMS error performances can be found in Fig. \ref{Fig:RMS}. As we can see, both the hybrid and dynamics observers exhibit a large RMS value under some driving situations. Algorithm 2 provides superior performances in all scenario tests.  
\begin{figure}[t]
	\begin{center}
		\includegraphics[width = 3.4 in]{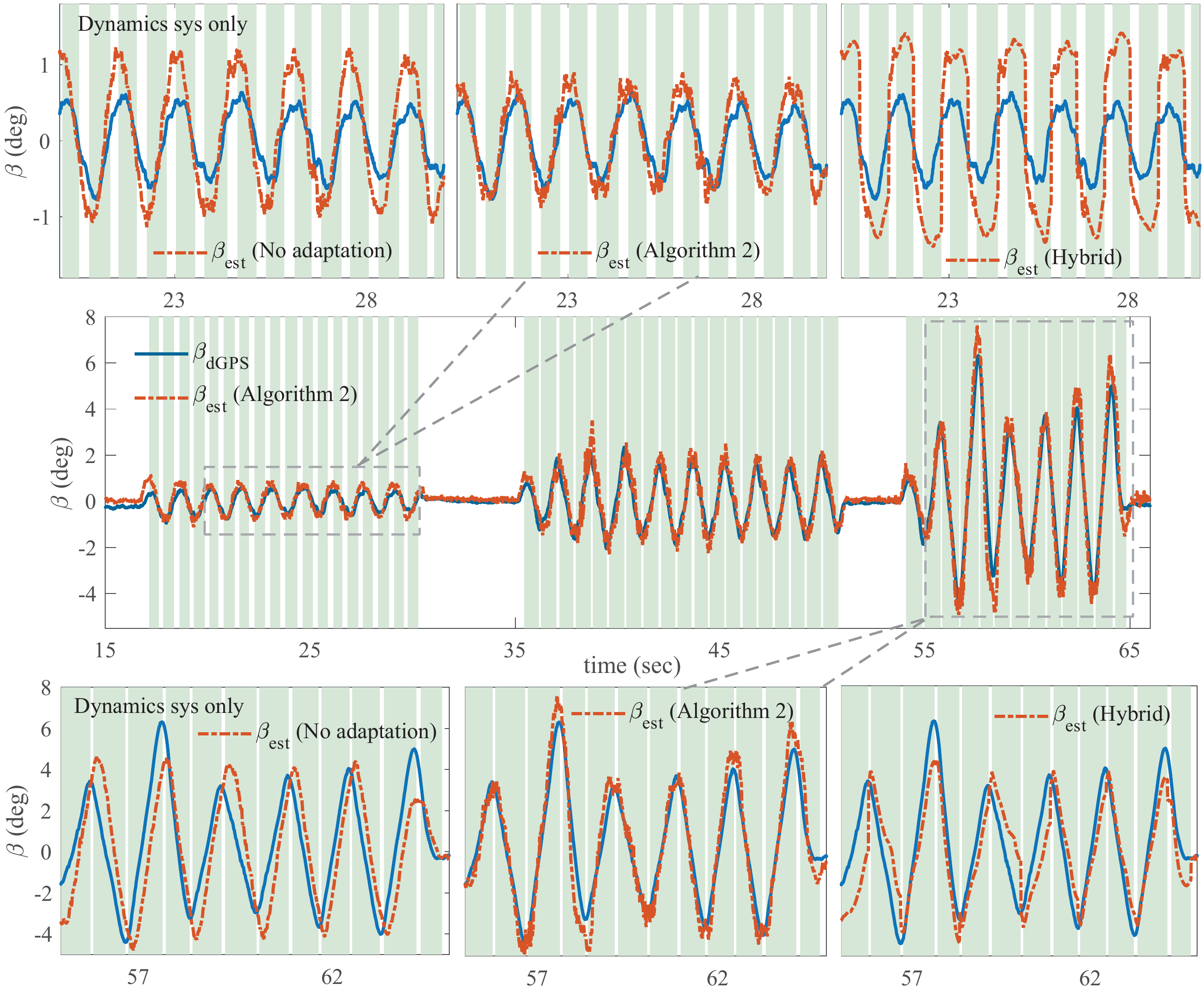}
		\caption{Comparison of the sideslip angle estimation for a slalom test.} 
		\label{Fig:Slalom}
	\end{center}
\end{figure}
\begin{figure}[t]
	\begin{center}
		\includegraphics[width = 3.5in]{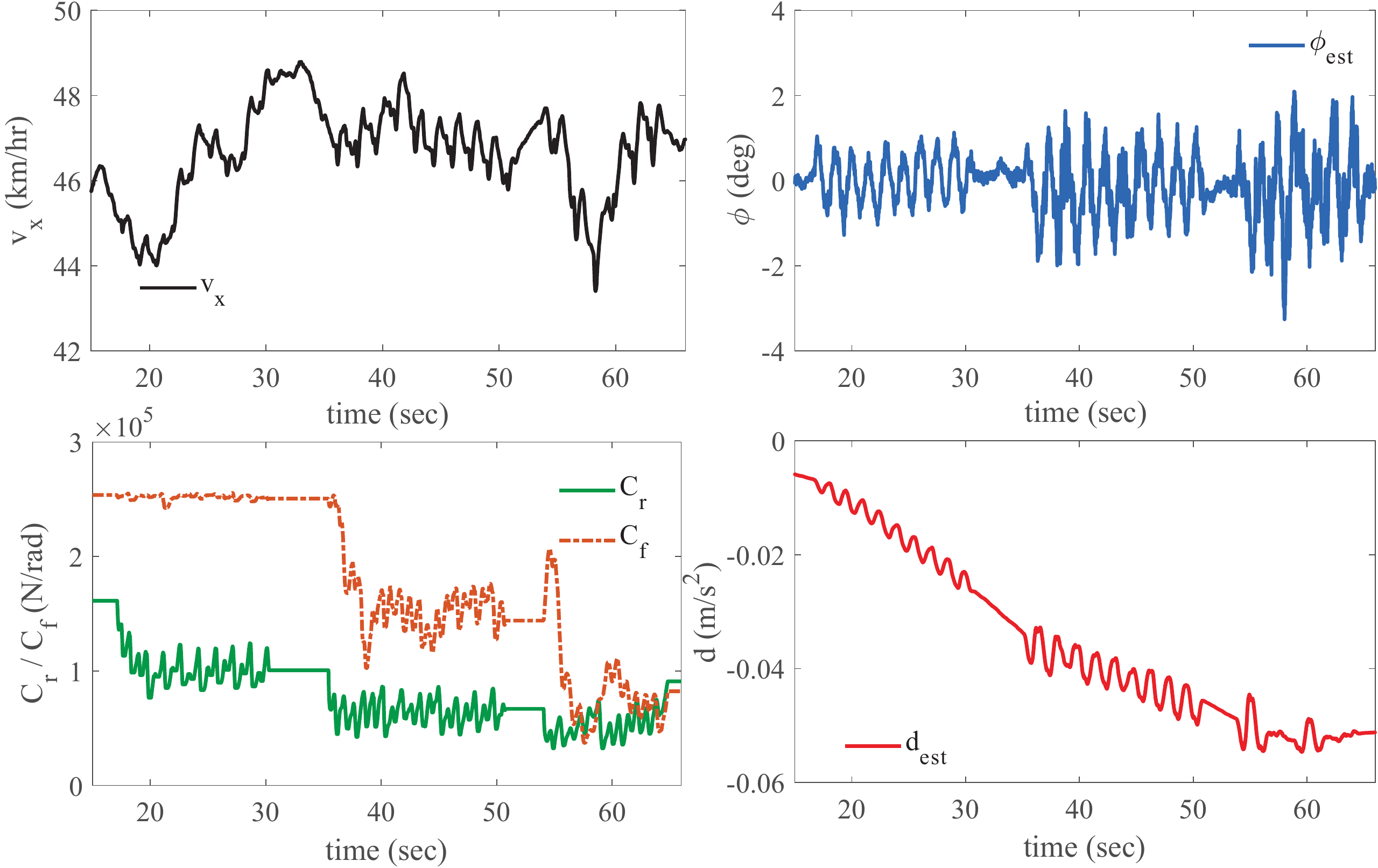}
		\caption{Slalom test results for Algorithm 2: longitudinal velocity; adapted cornering stiffnesses; estimated bank angle and sensor bias.}
		\label{Fig:Slalom2} 
	\end{center}
\end{figure}

Starting from slalom, severe single lane changing and steady circular motion tests, we can see that Fig. \ref{Fig:Slalom}, Fig. \ref{Fig:SLC} and Fig. \ref{Fig:cir} demonstrate the effectiveness of the proposed estimator. For the dynamics model-based approach, it is obvious that there is a big disparity between the true and the estimated sideslip angle when the vehicle enters the nonlinear tire region. As expected, for the method switching between dynamics and kinematics models, we can see a discontinuous estimating during the transition. The longitudinal velocity, adapted tire cornering stiffnesses, estimated bank angle and sensor bias for all scenario tests are shown in Fig. \ref{Fig:Slalom2}, Fig. \ref{Fig:SLC2} and Fig. \ref{Fig:cir2}. The adapted cornering stiffnesses becomes smaller for a low friction road condition or entering the nonlinear tire region. Since the estimated bank angle is affected by the vehicle roll angle, we can conclude that all the estimated bank angle resulting within -4$^\circ$- 4$^\circ$ may be questionable.
However, we still trust the estimation for large bank angles. Fig. \ref{Fig:bank1} and Fig. \ref{Fig:bank2} shows the experimental results of double lane change tests on a road with a large bank angle.
The estimate performance of the switching algorithm is poor since the kinematics model is sensitive to the lateral acceleration measurement disturbance introduced from the bank angle. Fig. \ref{Fig:bank2} confirms the ability of the proposed algorithm to estimating the bank angle, with the estimated value of the bank angle converging to the true value of 14$^\circ$. The adapted cornering stiffnesses remain unchanged because of the mild driving condition. 
Finally, we evaluate the performance of Algorithm 2 by conducting a stop-N-turn test for a varying low speed condition (Fig. \ref{Fig:stopnturn} and Fig. \ref{Fig:stopnturn2}). Again, the results are very promising.
\begin{figure}[t] 
\begin{center} 
		\includegraphics[width = 3.5 in]{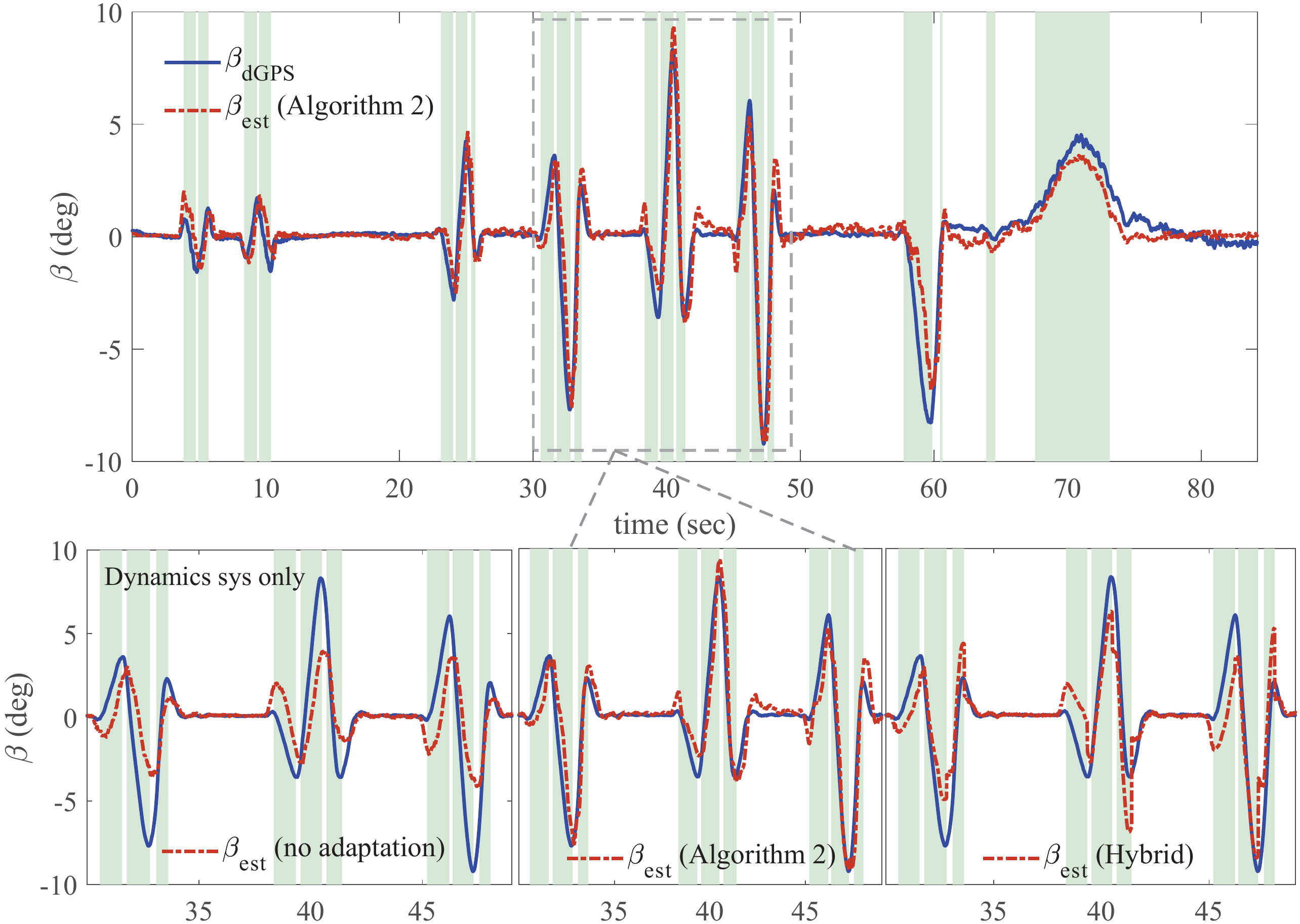}
		\caption{Comparison of the sideslip angle estimation for a severe single lane changing maneuver.} 
		\label{Fig:SLC}
	\end{center}
\end{figure}
\begin{figure}[t]
	\begin{center}
		\includegraphics[width = 3.5in]{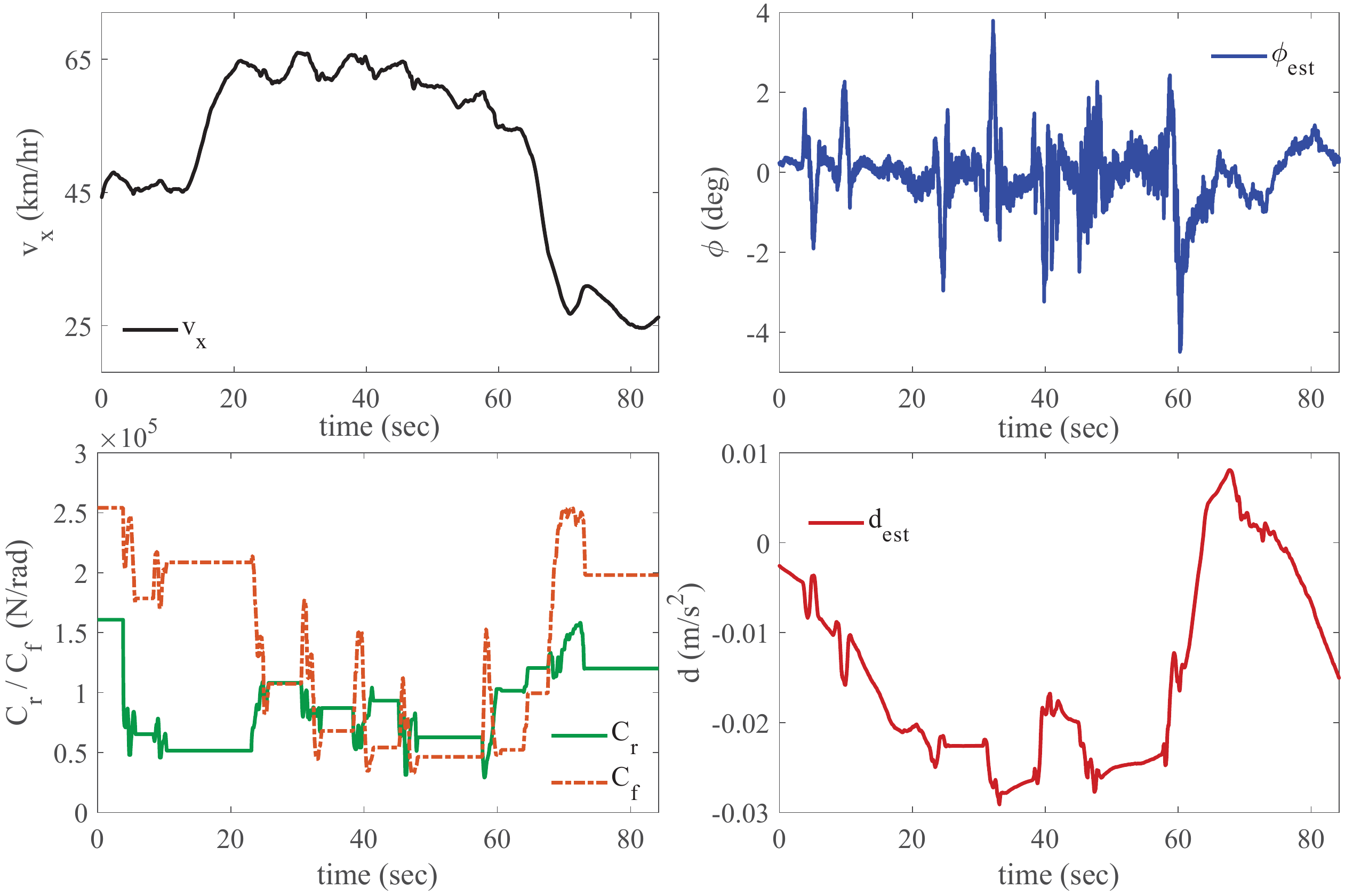}
		\caption{Severe single lane changing test results for Algorithm 2: longitudinal velocity; adapted cornering stiffnesses; estimated bank angle and sensor bias.} \label{Fig:SLC2}
	\end{center}
\end{figure} 

In summary, the proposed algorithm stands out for its robustness in model error and measurement disturbance. It can be used for any driving situation with different road conditions. In addition, reliable estimates for bank angle and sensor bias are also available.
\section{Conclusion}
This paper developed a real-time algorithm for estimation of sideslip angle using inexpensive sensors normally available for electronic stability control (ESC) applications. The algorithm utilizes a kinematics observer to improve the estimation based on a vehicle dynamics model. It also provides estimates of road bank angles, lateral acceleration sensor bias and tire cornering stiffness. 
The algorithm performance is evaluated through several experimental tests and the results indicate that the algorithm provides a good estimate of the vehicle sideslip angle both in normal and extreme maneuvers with different road conditions. 
\begin{figure}[t]
	\begin{center}
		\includegraphics[width = 3.4 in]{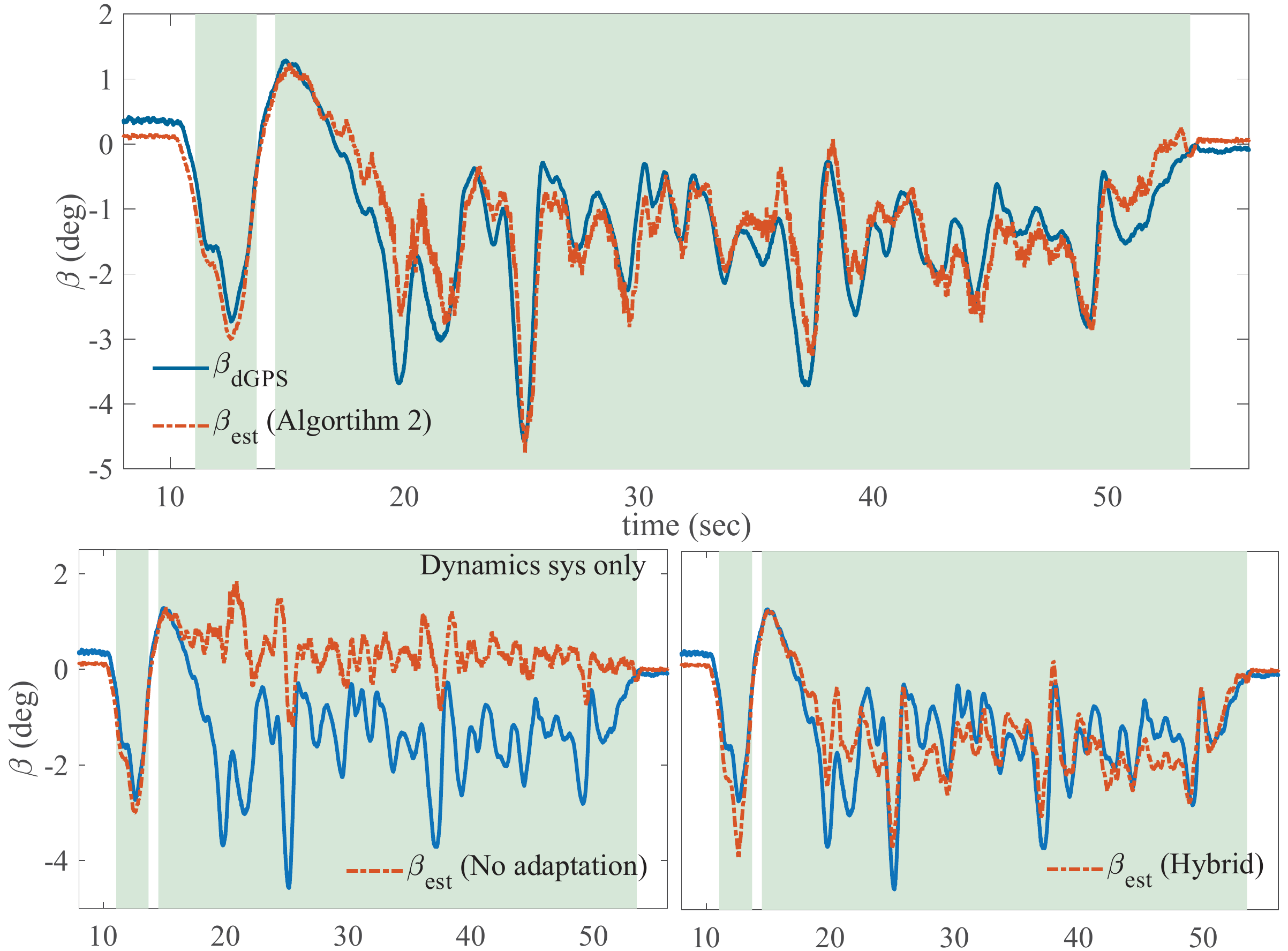}
		\caption{Comparison of the sideslip angle estimation for a steady circular motion.}
		\label{Fig:cir}
	\end{center}
\end{figure}
\begin{figure}[t]
	\begin{center}
		\includegraphics[width = 3.4in]{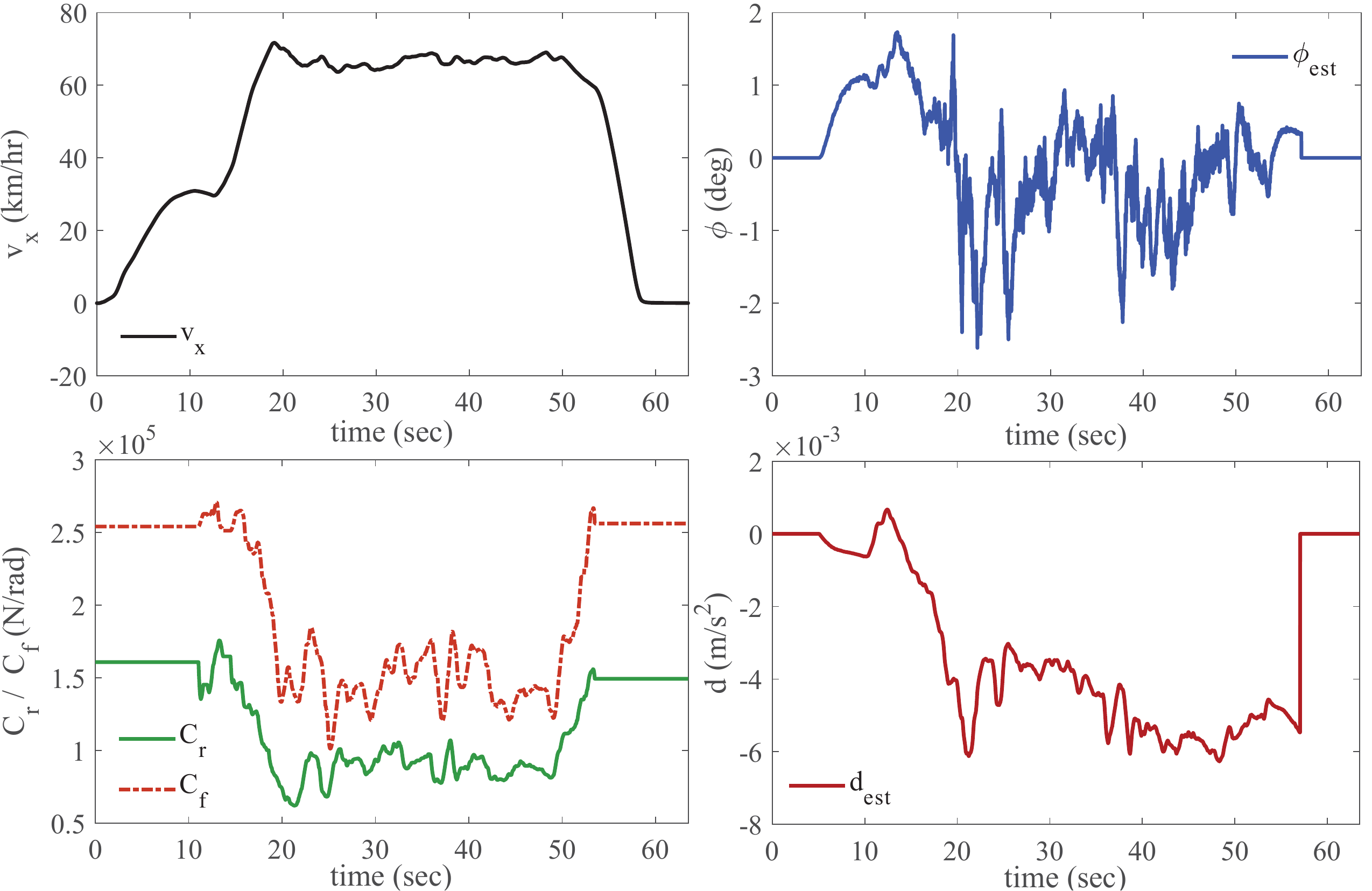}
		\caption{Steady circular motion test results for Algorithm 2: longitudinal velocity; adapted cornering stiffnesses; estimated bank angle and sensor bias.} 
		\label{Fig:cir2}
	\end{center}
\end{figure}
\begin{figure}[t]
	\begin{center}
		\includegraphics[width = 3.4 in]{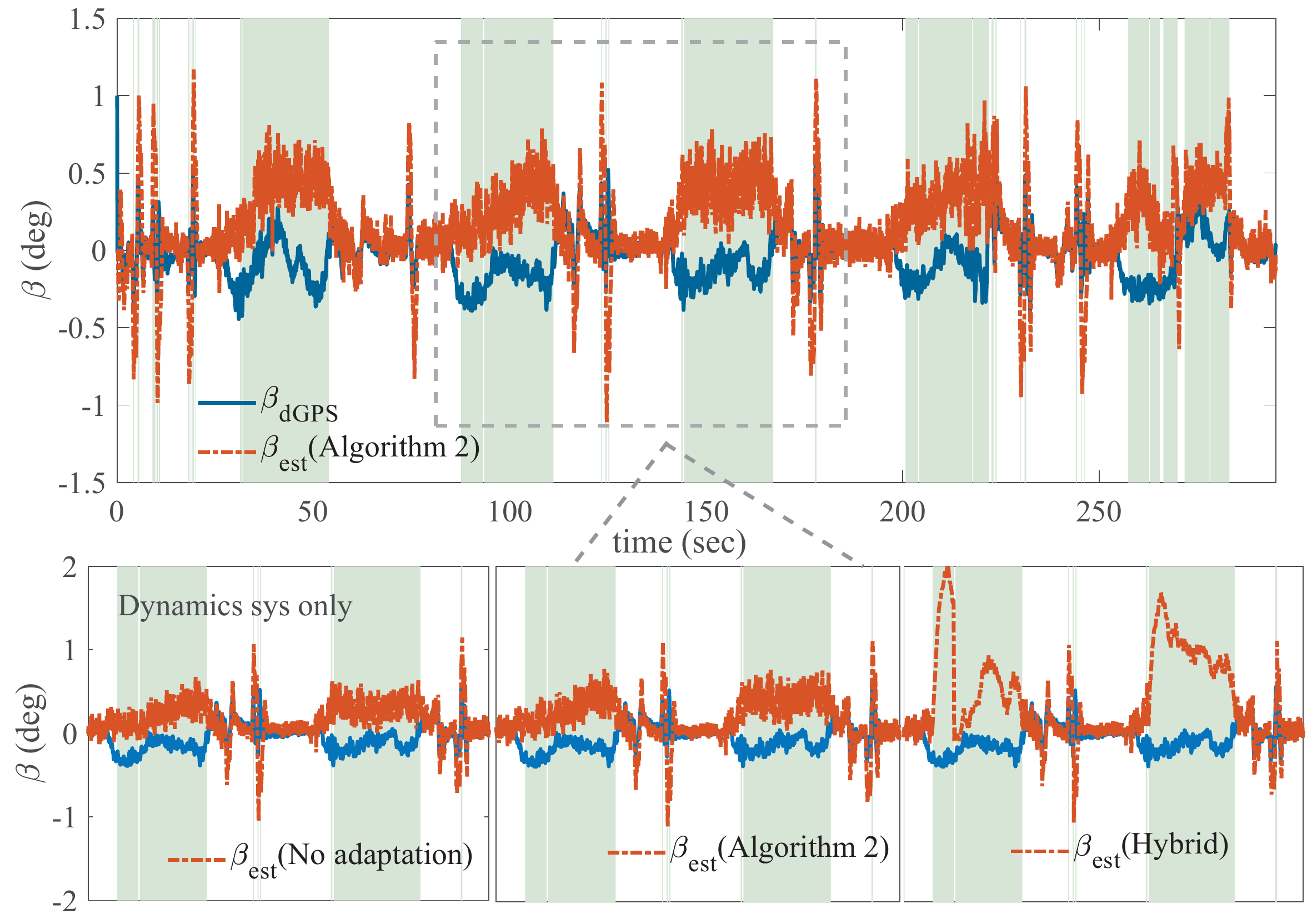}
		\caption{Comparison of the sideslip angle estimation for double lane changing on a bank.} 
		\label{Fig:bank1}
	\end{center}
\end{figure}
\begin{figure}[t]
	\begin{center}
		\includegraphics[width = 3.5in]{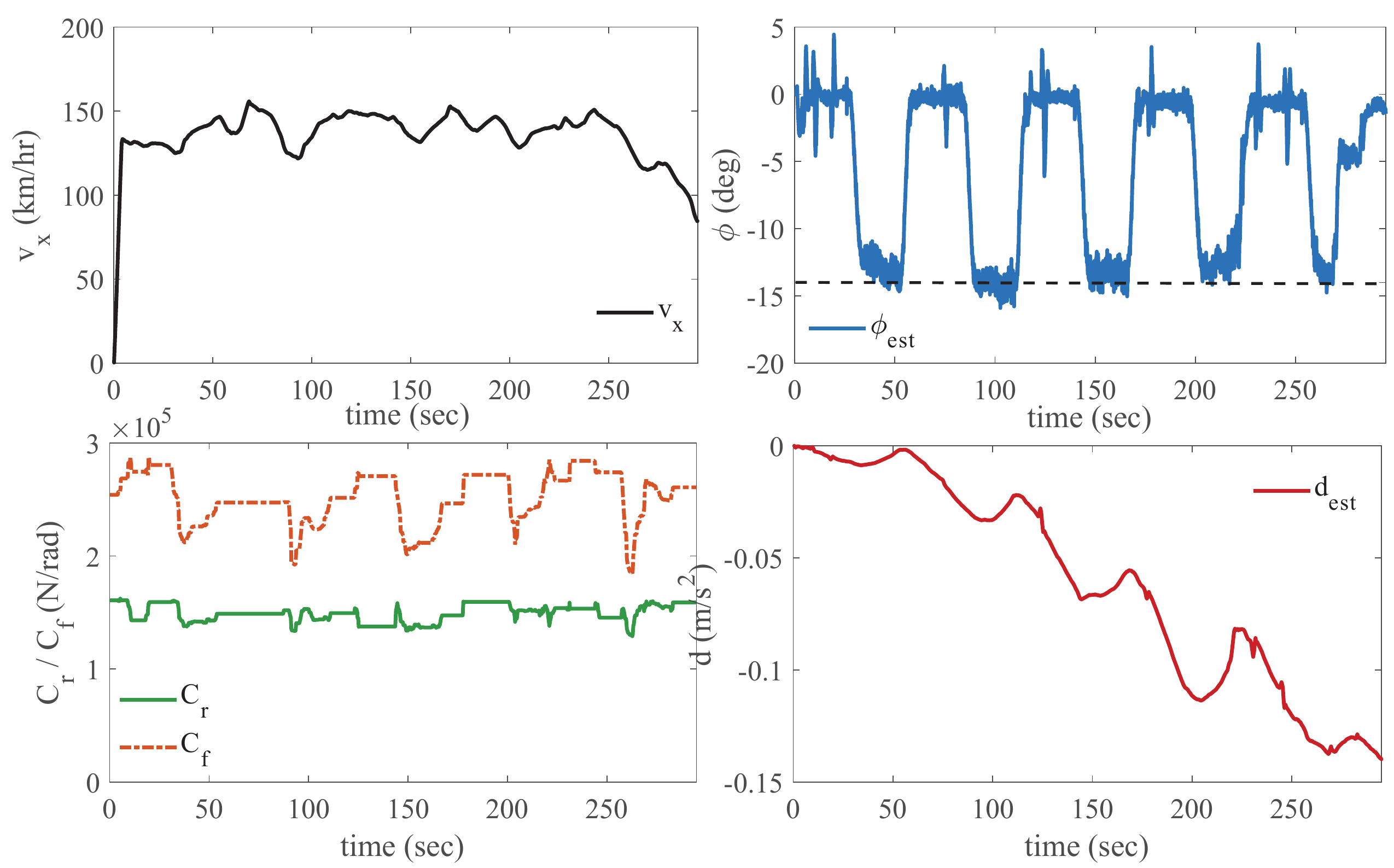}
		\caption{On-bank double lane changing test results for Algorithm 2: longitudinal velocity; adapted cornering stiffnesses; estimated bank angle and sensor bias.} 
		\label{Fig:bank2}
	\end{center}
\end{figure}
\begin{figure}[t]
	\begin{center}
		\includegraphics[width = 3.5 in]{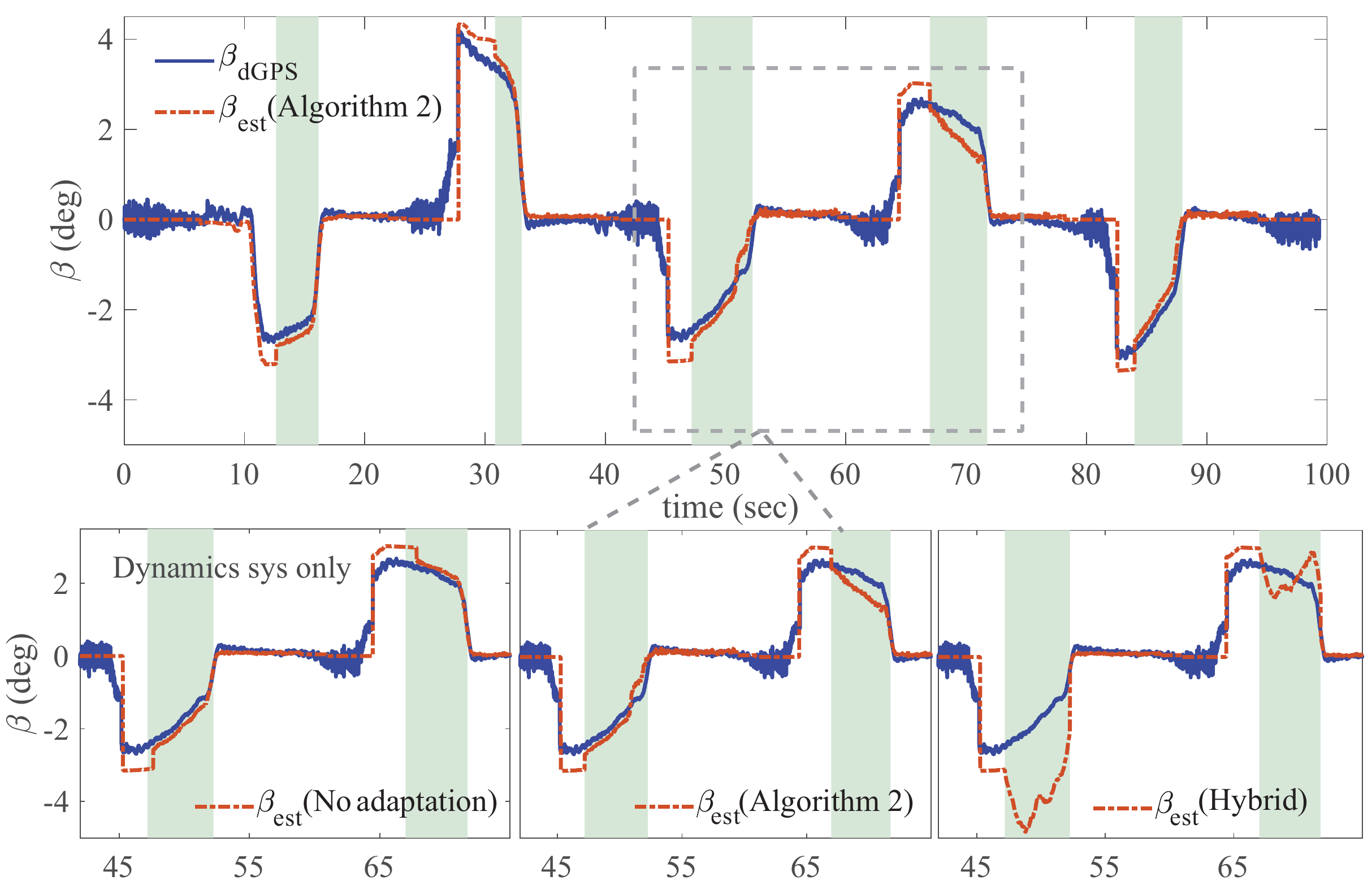}
		\caption{Comparison of the sideslip angle estimation for a stop-N-turn motion.} 
		\label{Fig:stopnturn}
	\end{center}
\end{figure}
\begin{figure}[t]
	\begin{center}
		\includegraphics[width = 3.5in]{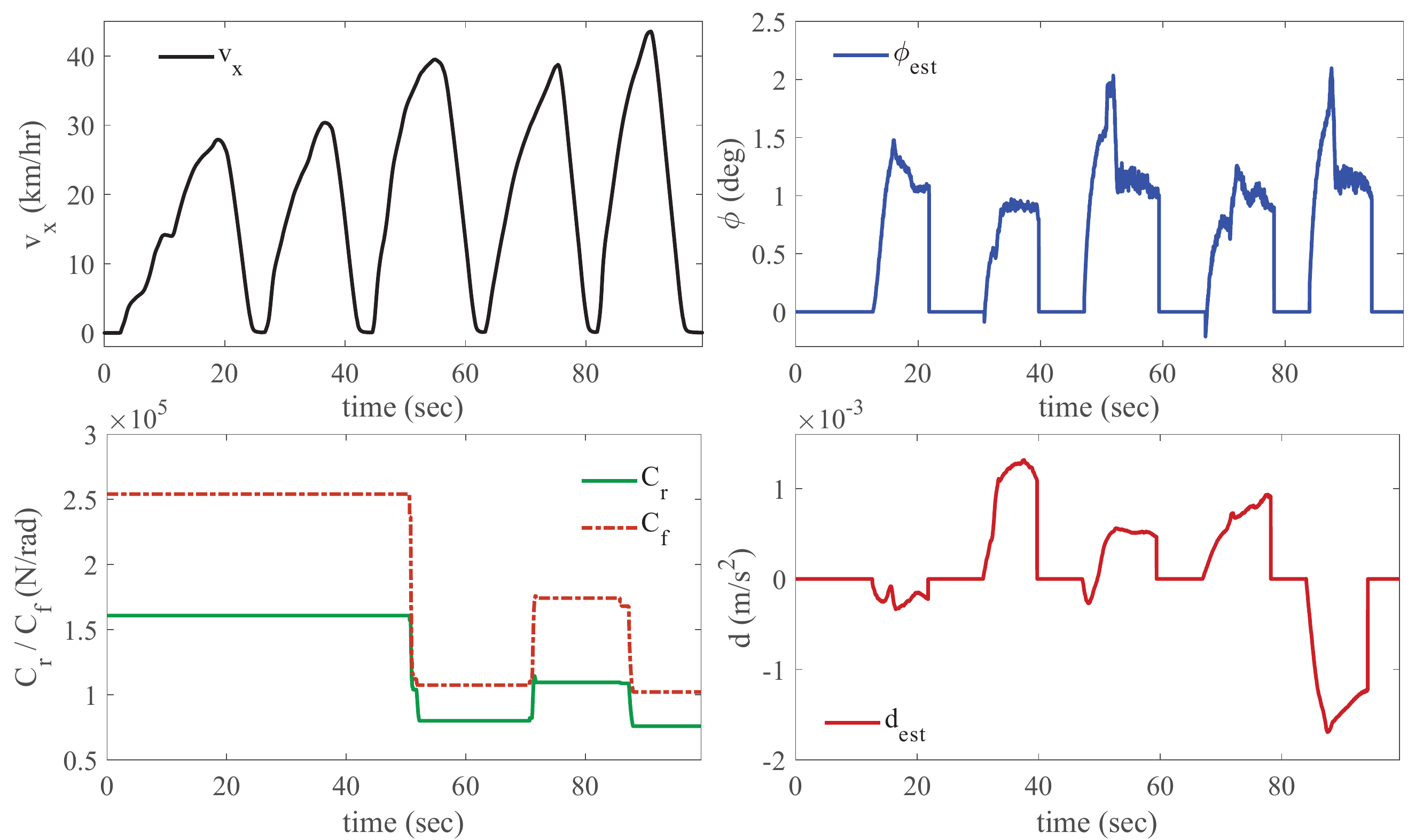}
		\caption{Stop-N-turn test results for Algorithm 2: longitudinal velocity; adapted cornering stiffnesses; estimated bank angle and sensor bias.} 
		\label{Fig:stopnturn2}
	\end{center}
\end{figure} 
\appendices
\section{Asymptotic Hyperstability and strictly positive real \cite{landau1987adaptive}} 
\begin{definition} 
	The feedback system shown in Fig. \ref{Fig:HyperStabilityBlock} is  asymptotically hyperstable if  the state $x_k$ of the linear time invariant system converges to zero  for $k\rightarrow\infty$
\end{definition}
\begin{theorem} 
	The feedback system shown in Fig. \ref{Fig:HyperStabilityBlock} is asymptotically hyperstable  if and only if 
	\begin{enumerate}
		\item the linear time invariant system is strictly positive real.
		\item the nonlinear feedback block satisfies Popov inequality:
		\begin{align}
		\exists~ \gamma > 0, \quad\sum_{k=1}^{k_1} w^T_kv_k \geq -\gamma^2 \quad \forall k_1\geq 0 \nonumber
		\end{align}
		\item the output signal, $w_k$, of the nonlinear block is bounded. 
	\end{enumerate}
\end{theorem}
\begin{theorem}
	A single input single output discrete-time system, $G(z)$, is strictly positive real if
	\begin{enumerate}
		\item[1.] the system does not possess any pole outside of or on the unit circle on z-plane.
		\item[2.] $\forall~|\omega|<\pi,~G(e^{-j\omega})+G(e^{j\omega}) > 0$
	\end{enumerate}
\end{theorem}
\section*{Acknowledgment}
The authors gratefully acknowledge the financial support received from the Hyundai Motor Company and the Hyundai-Kia Motors California Proving Ground for supporting this work. Thanks to Wanki Cho, Jinhwan Choi and Jongho Lee for the technical support.
%
\bibliographystyle{IEEEtran}
\bibliography{IEEEmyfile}

\begin{thebibliography}{10}
\providecommand{\url}[1]{#1}
\csname url@samestyle\endcsname
\providecommand{\newblock}{\relax}
\providecommand{\bibinfo}[2]{#2}
\providecommand{\BIBentrySTDinterwordspacing}{\spaceskip=0pt\relax}
\providecommand{\BIBentryALTinterwordstretchfactor}{4}
\providecommand{\BIBentryALTinterwordspacing}{\spaceskip=\fontdimen2\font plus
\BIBentryALTinterwordstretchfactor\fontdimen3\font minus
  \fontdimen4\font\relax}
\providecommand{\BIBforeignlanguage}[2]{{%
\expandafter\ifx\csname l@#1\endcsname\relax
\typeout{** WARNING: IEEEtran.bst: No hyphenation pattern has been}%
\typeout{** loaded for the language `#1'. Using the pattern for}%
\typeout{** the default language instead.}%
\else
\language=\csname l@#1\endcsname
\fi
#2}}
\providecommand{\BIBdecl}{\relax}
\BIBdecl

\bibitem{pilutti1998vehicle}
T.~Pilutti, G.~Ulsoy, and D.~Hrovat, ``Vehicle steering intervention through
  differential braking,'' \emph{Journal of dynamic systems, measurement, and
  control}, vol. 120, no.~3, pp. 314--321, 1998.

\bibitem{tseng1999development}
H.~E. Tseng, B.~Ashrafi, D.~Madau, T.~A. Brown, and D.~Recker, ``The
  development of vehicle stability control at ford,'' \emph{IEEE/ASME
  transactions on mechatronics}, vol.~4, no.~3, pp. 223--234, 1999.

\bibitem{hara1991traction}
M.~Hara, S.~Kamio, M.~Takao, K.~Sakita, and T.~Abe, ``Traction control
  system,'' May~28 1991, uS Patent 5,018,595.

\bibitem{wilson1998driver}
R.~Wilson-Jones, R.~H. A.~H. Tribe, and M.~Appleyard, ``Driver assistance
  system for a vehicle,'' Jun.~9 1998, uS Patent 5,765,116.

\bibitem{lee2004robust}
S.~Lee and J.-B. Song, ``Robust mobile robot localization using optical flow
  sensors and encoders,'' in \emph{Robotics and Automation, 2004. Proceedings.
  ICRA'04. 2004 IEEE International Conference on}, vol.~1.\hskip 1em plus 0.5em
  minus 0.4em\relax IEEE, 2004, pp. 1039--1044.

\bibitem{bevly2006integrating}
D.~M. Bevly, J.~Ryu, and J.~C. Gerdes, ``Integrating ins sensors with gps
  measurements for continuous estimation of vehicle sideslip, roll, and tire
  cornering stiffness,'' \emph{IEEE Transactions on Intelligent Transportation
  Systems}, vol.~7, no.~4, pp. 483--493, 2006.

\bibitem{bevly2000use}
D.~M. Bevly, J.~C. Gerdes, C.~Wilson, and G.~Zhang, ``The use of gps based
  velocity measurements for improved vehicle state estimation,'' in
  \emph{American Control Conference, 2000. Proceedings of the 2000},
  vol.~4.\hskip 1em plus 0.5em minus 0.4em\relax IEEE, 2000, pp. 2538--2542.

\bibitem{selmanaj2017robust}
D.~Selmanaj, M.~Corno, G.~Panzani, and S.~Savaresi, ``Robust vehicle sideslip
  estimation based on kinematic considerations,'' \emph{IFAC-PapersOnLine},
  vol.~50, no.~1, pp. 14\,855--14\,860, 2017.

\bibitem{farrelly1996estimation}
J.~Farrelly and P.~Wellstead, ``Estimation of vehicle lateral velocity,'' in
  \emph{Control Applications, 1996., Proceedings of the 1996 IEEE International
  Conference on}.\hskip 1em plus 0.5em minus 0.4em\relax IEEE, 1996, pp.
  552--557.

\bibitem{strano2018constrained}
S.~Strano and M.~Terzo, ``Constrained nonlinear filter for vehicle sideslip
  angle estimation with no a priori knowledge of tyre characteristics,''
  \emph{Control Engineering Practice}, vol.~71, pp. 10--17, 2018.

\bibitem{dakhlallah2008tire}
J.~Dakhlallah, S.~Glaser, S.~Mammar, and Y.~Sebsadji, ``Tire-road forces
  estimation using extended kalman filter and sideslip angle evaluation,'' in
  \emph{American Control Conference, 2008}.\hskip 1em plus 0.5em minus
  0.4em\relax IEEE, 2008, pp. 4597--4602.

\bibitem{aoki2004robust}
Y.~Aoki, T.~Inoue, and Y.~Hori, ``Robust design of gain matrix of body slip
  angle observer for electric vehicles and its experimental demonstration,'' in
  \emph{Advanced Motion Control, 2004. AMC'04. The 8th IEEE International
  Workshop on}.\hskip 1em plus 0.5em minus 0.4em\relax IEEE, 2004, pp. 41--45.

\bibitem{li2014variable}
L.~Li, G.~Jia, X.~Ran, J.~Song, and K.~Wu, ``A variable structure extended
  kalman filter for vehicle sideslip angle estimation on a low friction road,''
  \emph{Vehicle System Dynamics}, vol.~52, no.~2, pp. 280--308, 2014.

\bibitem{shao2016nonlinear}
L.~Shao, C.~Jin, C.~Lex, and A.~Eichberger, ``Nonlinear adaptive observer for
  side slip angle and road friction estimation,'' in \emph{Decision and Control
  (CDC), 2016 IEEE 55th Conference on}.\hskip 1em plus 0.5em minus 0.4em\relax
  IEEE, 2016, pp. 6258--6265.

\bibitem{you2009new}
S.-H. You, J.-O. Hahn, and H.~Lee, ``New adaptive approaches to real-time
  estimation of vehicle sideslip angle,'' \emph{Control Engineering Practice},
  vol.~17, no.~12, pp. 1367--1379, 2009.

\bibitem{liu1998state}
C.-S. Liu and H.~Peng, ``A state and parameter identification sclieme for
  linearly parameterized systems,'' 1998.

\bibitem{grip2009vehicle}
H.~F. Grip, L.~Imsland, T.~A. Johansen, J.~C. Kalkkuhl, and A.~Suissa,
  ``Vehicle sideslip estimation,'' \emph{IEEE control systems}, vol.~29, no.~5,
  2009.

\bibitem{coy2014decision}
J.~L. Coyte, B.~Li, H.~Du, W.~Li, D.~Stirling, and M.~Ros, ``Decision tree
  assisted ekf for vehicle slip angle estimation using inertial motion
  sensors,'' in \emph{Neural Networks (IJCNN), 2014 International Joint
  Conference on}.\hskip 1em plus 0.5em minus 0.4em\relax IEEE, 2014, pp.
  940--946.

\bibitem{peng1experimental}
A.~Y. U.~H. Peng and H.~Tseng, ``Experimental verification of lateral speed
  estimation methods,'' \emph{a a}, vol.~1.

\bibitem{boada2016vehicle}
B.~Boada, M.~Boada, and V.~Diaz, ``Vehicle sideslip angle measurement based on
  sensor data fusion using an integrated anfis and an unscented kalman filter
  algorithm,'' \emph{Mechanical Systems and Signal Processing}, vol.~72, pp.
  832--845, 2016.

\bibitem{kang2016vehicle}
C.~M. Kang, S.-H. Lee, and C.~C. Chung, ``Vehicle lateral motion estimation
  with its dynamic and kinematic models based interacting multiple model
  filter,'' in \emph{Decision and Control (CDC), 2016 IEEE 55th Conference
  on}.\hskip 1em plus 0.5em minus 0.4em\relax IEEE, 2016, pp. 2449--2454.

\bibitem{chen2008sideslip}
B.-C. Chen and F.-C. Hsieh, ``Sideslip angle estimation using extended kalman
  filter,'' \emph{Vehicle System Dynamics}, vol.~46, no.~S1, pp. 353--364,
  2008.

\bibitem{piyabongkarn2009development}
D.~Piyabongkarn, R.~Rajamani, J.~A. Grogg, and J.~Y. Lew, ``Development and
  experimental evaluation of a slip angle estimator for vehicle stability
  control,'' \emph{IEEE Transactions on Control Systems Technology}, vol.~17,
  no.~1, pp. 78--88, 2009.

\bibitem{de2017real}
M.~De~Martino, F.~Farroni, N.~Pasquino, A.~Sakhnevych, and F.~Timpone,
  ``Real-time estimation of the vehicle sideslip angle through regression based
  on principal component analysis and neural networks,'' in \emph{Systems
  Engineering Symposium (ISSE), 2017 IEEE International}.\hskip 1em plus 0.5em
  minus 0.4em\relax IEEE, 2017, pp. 1--6.

\bibitem{ungoren2004study}
A.~Y. Ungoren, H.~Peng, and H.~Tseng, ``A study on lateral speed estimation
  methods,'' \emph{International Journal of Vehicle Autonomous Systems},
  vol.~2, no. 1-2, pp. 126--144, 2004.

\bibitem{rajamani2011vehicle}
R.~Rajamani, \emph{Vehicle dynamics and control}.\hskip 1em plus 0.5em minus
  0.4em\relax Springer Science \& Business Media, 2011.

\bibitem{Tibshirani94regressionshrinkage}
R.~Tibshirani, ``Regression shrinkage and selection via the lasso,''
  \emph{Journal of the Royal Statistical Society, Series B}, vol.~58, pp.
  267--288, 1994.

\bibitem{eksioglu2011rls}
E.~M. Eksioglu and A.~K. Tanc, ``Rls algorithm with convex regularization,''
  \emph{IEEE Signal Processing Letters}, vol.~18, no.~8, pp. 470--473, 2011.

\bibitem{landau1990system}
I.~D. Landau, ``System identification and control design,'' 1990.

\bibitem{song1992extended}
Y.~Song and J.~W. Grizzle, ``The extended kalman filter as a local asymptotic
  observer for nonlinear discrete-time systems,'' in \emph{American Control
  Conference, 1992}.\hskip 1em plus 0.5em minus 0.4em\relax IEEE, 1992, pp.
  3365--3369.

\bibitem{anderson1968simplified}
B.~Anderson, ``A simplified viewpoint of hyperstability,'' \emph{IEEE
  Transactions on Automatic Control}, vol.~13, no.~3, pp. 292--294, 1968.

\bibitem{khalil1996noninear}
H.~K. Khalil, ``Noninear systems,'' \emph{Prentice-Hall, New Jersey}, vol.~2,
  no.~5, pp. 5--1, 1996.

\bibitem{landau1987adaptive}
Y.~Landau, ``Adaptive control: the model reference approach, 1979,''
  \emph{Marcel Dekker, New York). Hac, A. Adaptive control of vehicle
  suspension. Veh. System Dynamics}, vol.~16, pp. 74--77, 1987.

\end{thebibliography}
\begin{IEEEbiography}[
{
\includegraphics[width=1in,height=1.25in,clip,keepaspectratio]{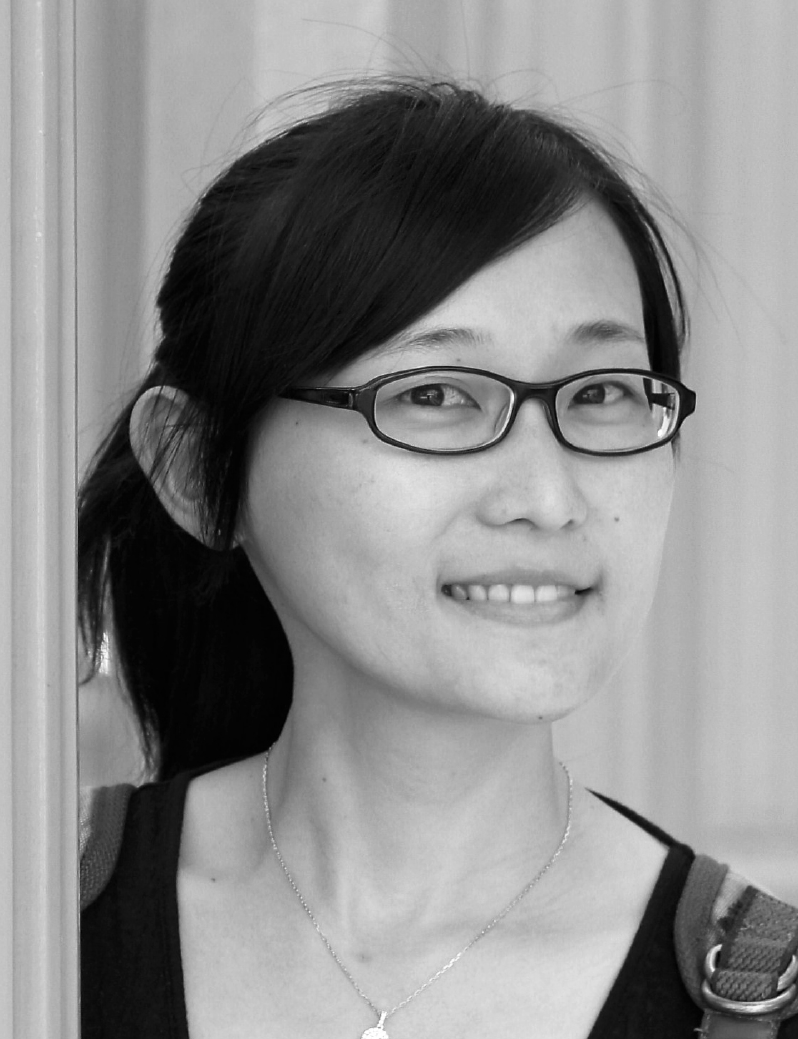}
}
]{Yi-Wen Liao} 
received the B.S. degree in mechanical engineering from National Taiwan University, Taipei, Taiwan, in 2010, and the M.S. degree in mechanical engineering from the University of Michigan, Ann Arbor, MI, USA, in 2012. She is currently working toward the Ph. D. degree in the Department of Mechanical Engineering at the Model Predictive Control Laboratory, University of California, Berkeley, headed by Prof. Francesco Borrelli. 

Her research interests include robust model predictive control, adaptive nonlinear control, vehicle dynamics, and their applications to autonomous driving systems. 
\end{IEEEbiography}
\begin{IEEEbiography}[
{
\includegraphics[width=1in,height=1.25in,clip]{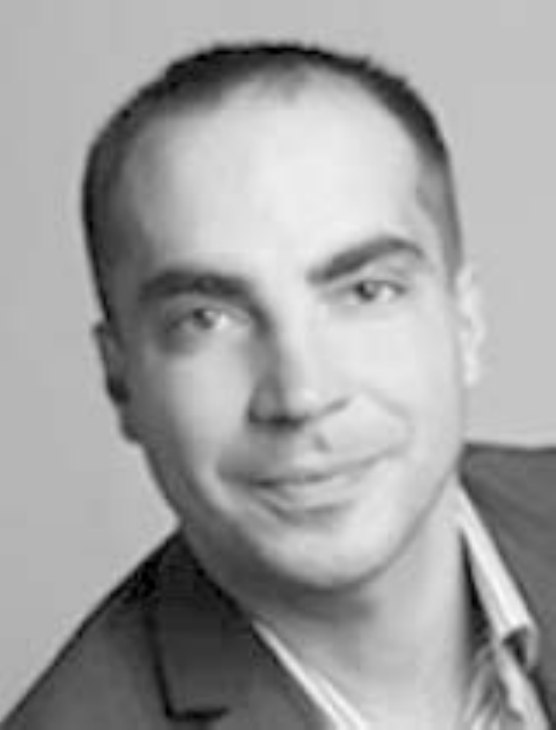}
}
]{Francesco Borrelli} 
received the Laurea degree in computer science engineering from the University of Naples Federico II, Naples, Italy, in 1998, and the Ph.D. degree from ETH-Zurich, Zurich, Switzerland, in 2002. He is currently a Professor with the Department of Mechanical Engineering, University of California, Berkeley, CA, USA. He is the author of more than 100 publications in the field of predictive control and author of the book Constrained Optimal Control of Linear and Hybrid
Systems (Springer-Verlag). 

His research interests include constrained
optimal control, model predictive control and its application to
advanced automotive control and energy efficient building operation. Dr. Borrelli received the 2009 National Science Foundation CAREER
Award and the 2012 IEEE Control System Technology Award. In 2008,
he became the Chair of the IEEE Technical Committee on Automotive
Control.
\end{IEEEbiography}
\vfill 
%
%
\end{document}